\documentclass[11pt]{article}
\usepackage{amsmath}
\usepackage{amstext,amsbsy}
\usepackage{epsfig,multicol}
\usepackage{amsthm}
\usepackage{amssymb,latexsym}
\usepackage{color}
\usepackage{tikz}
\usepackage{enumerate}
\usepackage{framed}
\usepackage{caption}
\usepackage{subcaption}
\usetikzlibrary{decorations.pathreplacing}

\newtheorem{thm}{Theorem}
\newtheorem{defi}[thm]{Definition}
\newtheorem{exa}[thm]{Example}
\newtheorem{lem}[thm]{Lemma}
\newtheorem{cor}[thm]{Corollary}

\def \se {\mathrm{se}} 
\def \ne {\mathrm{ne}}  

\def \M {\mathcal{M}}
\def \N {\mathcal{N}}
\def \R {\mathcal{R}}

\def \F {\mathcal{F}}
\def \S {\mathcal{S}}

\begin{document}
\title{Maximal increasing sequences in fillings of almost-moon polyominoes }
\author{
Svetlana Poznanovi\'c$^1$ and  Catherine H.
Yan$^2$ \\ [6pt]
$^{1}$ Department of Mathematical Sciences\\
Clemson University, Clemson, SC 29634  \\ [6pt] 
$^{2}$Department of Mathematics\\
Texas A\&M University, College Station, TX 77843, USA\\[5pt]
}
\date{} 
\maketitle
\begin{abstract}
It was proved by Rubey that the number of fillings with zeros and ones of a given moon 
polyomino that do not contain a northeast chain of  size $k$  depends only 
on the set of columns of the polyomino, but not the shape of the polyomino. Rubey's proof is an 
adaption of \emph{jeu de taquin } and promotion for arbitrary fillings of moon polyominoes. 
In this paper we present a bijective proof for this result by considering fillings of 
almost-moon polyominoes, which are moon polyominoes after removing one of the rows. Explicitly, 
we construct a bijection which preserves the  size of the largest northeast chains of the fillings when two adjacent  
rows of the polyomino are exchanged. This bijection also preserves the column sum of the fillings. 
We also present a bijection that preserves  the size of the largest northeast chains, the row sum and the column sum if every 
row of the fillings has at most one $1$. 
\end{abstract}



{\renewcommand{\thefootnote}{} \footnote{\emph{E-mail addresses}:
spoznan@clemson.edu (S.~Poznanovi\'c), cyan@math.tamu.edu (C.H.~Yan) }

\footnotetext[1]{The first author is partially supported by NSF grant DMS-1312817. 
The second author is partially supported by NSF grant DMS-1161414.} 

\section{Introduction} \label{S:introduction}


The systematic study of  matchings and set partitions with certain restrictions on their crossings 
and nestings started in~\cite{CDDSY}, where Chen et al. used Robinson-Schensted-like insertion/deletion processes to 
show the symmetry between the sizes  of the largest crossings and the largest nestings.  Lying in the heart of \cite{CDDSY} is 
Greene's Theorem on the relation between increasing and decreasing subsequences in permutations and the shape of the tableaux 
which are obtained by the Robinson-Schensted correspondence.  These results have been put in a larger context of enumeration 
of fillings of polyominoes where one imposes restrictions on the increasing and decreasing chains of the fillings.  

Krettenthaler \cite{Krattenthaler} extended the results of \cite{CDDSY} to 01-fillings as well as $\mathbb{N}$-fillings  of Ferrers diagrams, and 
obtained various generalizations using Fomin's growth diagrams \cite{Fomin86,Fomin94,Fomin95}.  Jonsson,  motivated by the
problem of counting 
generalized triangulations with a given size of the maximal crossings, proved that 
the number of 01-fillings of a stack polyomino that do not contain a northeast chain of size  $k$ depends only on the 
distribution of lengths of the columns of the polyomino.   A stack polyomino is a convex polyomino in which the rows   are arranged in a  
descending order  from top to bottom. 
Jonsson \cite{Jon05} proved the result first for maximal 01-fillings using an involved induction,  and
later with Welker \cite{JW07}  for all 
01-fillings with  a fixed number of 1's using the machinery of simplicial complexes and commutative algebra.  
Rubey  \cite{Rubey} generalized Jonsson and Welker's result to moon 
polyominoes, using an adaptation of \emph{jeu de taquin} and promotion for arbitrary fillings of moon polyominoes.  

Moon polyominoes are polyominoes that are convex and intersection-free. In moon polyomino the lengths of rows (and columns) are 
arranged in a unimodal order.   For a moon polyomino $\M$, let $\sigma \M$ be another moon polyomino obtained 
by  permuting the rows of $\M$.    Rubey proved that
the number of 01-fillings with the longest northeast chains of size $k$ and  exactly $c_i$ non-zero entries in column $i$ 
are equal for $\M$ and $\sigma\M$. 
It is a very interesting property for  fillings of moon polyominoes: 
many combinatorial statistics are invariant under permutations of rows (or columns). In addition to Rubey's results, 
it is also known for a major index introduced by Chen, Poznanovic, Yan and Yang \cite{CPYY}, 
 for the numbers of northeast and southeast chains of length 2  by Kasraoui \cite{Kasra10}, and for various 
analogs and generalizations of 2-chains  \cite{CWYZ,WY13} .

The purpose of the present paper is to construct bijective proofs of Rubey's result. 
 Inspired by the work on layer polyominoes \cite{PYY13},  we  seek to extend the moon polyominoes 
to a general family that would  allow us to transform the moon polyomino $\M$ to $\sigma \M$ by a sequence of steps that 
interchange two adjacent rows at each time.   For this purpose we introduce the notion of almost-moon polyominoes, which 
become moon polyominoes after removing one of its rows (see Section 2 for  the exact definition).  
Let $\M$ and $\N$ be two almost-moon polyominoes that are related by an interchange of two adjacent rows.  We present two
bijections. The first is a map $\phi_{\M,\N}$   from 01-fillings of $\M$ with exactly $n$ 1's  to those of $\N$ such that 
it preserves the size of the longest northeast chains and the column sum.  The second map $\psi_{\M,\N}$ is restricted to 
fillings where every row has at most one $1$, and preserves the size of the longest northeast chains, the row 
sum,  and the column sum. 

Rubey's result implies that the number of 01-fillings of a moon polynomial without northeast chains of size $k$ equals those 
without southeast chains of size $k$. There are also known combinatorial transformations and bijections between certain families  
of fillings that avoid northeast chains of size $k$ and southeast chains of size $k$, for example,  by Backelin, West 
and Xin \cite{BWX07}  for 01-fillings of Ferrers diagrams  where every row and every column has exactly one $1$, and 
by de Mier \cite{deMier}  for $\mathbb{N}$-fillings of Ferrers diagrams with fixed row sum and column sum.  Nevertheless, 
none of them gives a bijection 
on the fillings of polyominoes that preserves the size of the longest northeast chains.

The paper is organized as follows. Section 2 contains necessary notations and the statement of main results. 
In Section 3 we construct the bijection $\phi_{\M,\N}$ in fillings of almost-moon polyominoes which preserves the size of maximal 
northeast chains and the columns sum.  In section 4 we restrict to  fillings that have at most one 1 in each row, and describe
the bijection $\psi_{\M,\N}$. We conclude 
the paper with some comments and counterexamples to a few seeming natural generalizations in Section 5.

\section{Notation and statements of the main results} 

A \emph{polyomino} is a finite subset of  $\mathbb{Z}^2$, where we represent every element $(i,j)$ of
 $\mathbb{Z}^2$ by a square cell. The polyomino is \emph{row-convex} (\emph{column-convex}) if
 its every row (column) is connected. If the polyomino is both row- and column-convex, we say
 that it is \emph{convex}. It is \emph{intersection-free} if every two columns are \emph{comparable}, 
i.e., the row-coordinates of one column form a subset of those of the other column. Equivalently,
 it is intersection-free if every two rows are comparable. A \emph{moon polyomino} is a convex 
intersection-free polyomino (e.g. Figure~\ref{fig:moon}).
The length of a row (or a column) is the number of cells in it. 
Note that in a moon polyomino the lengths of rows from top to bottom form a unimodal sequence.  
We will say that a row $\R$ is an  \emph{exceptional row}  of a  polyomino $\M$ 
  if there are rows above $\R$ and below $\R$ with  larger lengths in $\M$.  
 An \emph{almost-moon polyomino} is a polyomino with comparable convex rows and
 at most one exceptional row  (e.g. Figure~\ref{fig:almostmoon}). Therefore, every moon polyomino is 
also an almost-moon polyomino and an almost-moon 
polyomino is not necessarily column-convex. 

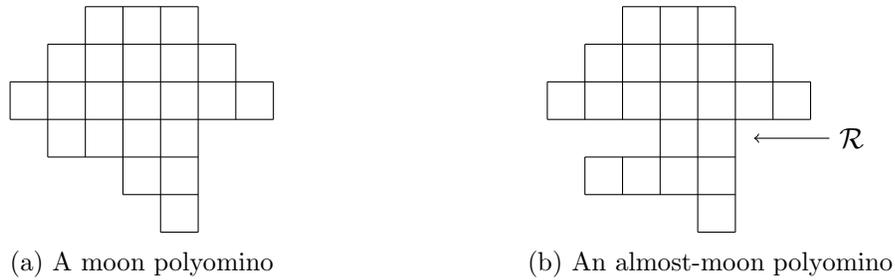
\begin{figure}[ht]
\centering
\begin{subfigure}[b]{0.45\textwidth}
\centering
  \begin{tikzpicture}
\draw (0, 1.5) -- (0,2);
\draw (0.5, 1) -- (0.5,2.5);
\draw (1, 1) -- (1,3);
\draw (1.5, 0.5) -- (1.5,3);
\draw (2, 0) -- (2,3);
\draw (2.5, 0) -- (2.5,3);
\draw (3, 1.5) -- (3,2.5);
\draw (3.5, 1.5) -- (3.5,2);

\draw (2,0) --(2.5,0);
\draw (1.5,0.5) --(2.5,0.5);
\draw (0.5,1) --(2.5,1);
\draw (0,1.5) --(3.5,1.5);
\draw (0,2) --(3.5,2);
\draw (0.5,2.5) --(3,2.5);
\draw (1,3) --(2.5,3);
\end{tikzpicture}
\caption{A moon polyomino}
\label{fig:moon}
\end{subfigure}
\begin{subfigure}[b]{0.45\textwidth}
\centering
\begin{tikzpicture}
\draw (0, 1.5) -- (0,2);
\draw (0.5, 0.5) -- (0.5,1); \draw (0.5,1.5) -- (0.5,2.5);
\draw (1, 0.5) -- (1,1); \draw (1, 1.5) -- (1, 3);
\draw (1.5, 0.5) -- (1.5,3);
\draw (2, 0) -- (2,3);
\draw (2.5, 0) -- (2.5,3);
\draw (3, 1.5) -- (3,2.5);
\draw (3.5, 1.5) -- (3.5,2);

\draw (2,0) -- (2.5,0);
\draw (0.5,0.5) --(2.5,0.5);
\draw (0.5,1) --(2.5,1);
\draw (0,1.5) --(3.5,1.5);
\draw (0,2) --(3.5,2);
\draw (0.5,2.5) --(3,2.5);
\draw (1,3) --(2.5,3);
\draw[->] (3.75, 1.25) node[right] {$\mathcal{R}$} -- (2.75, 1.25); 
\end{tikzpicture}
\caption{An almost-moon polyomino}
\label{fig:almostmoon}
\end{subfigure}
\caption{A moon polyomino and an almost-moon polyomino with an exceptional row $\R$ that differ by an
 interchange of adjacent rows. }
\label{fig:moonexample}
\end{figure}

In this paper we will consider polyominoes whose cells are filled with zeros and ones.
 A \emph{northeast chain}, or shortly \emph{ne-chain}, of size $k$ in such a filling is a
 set of $k$ cells $\{(i_1, j_1), (i_2, j_2), \dots, (i_k, j_k)\}$ with $i_1 < \cdots < i_k$, $ j_1 < \cdots < j_k$
 filled with 1's such that the $k \times k$ submatrix 
\[   \mathcal{G} = \{ (i_r, j_s) : 1 \leq r \leq k, 1 \leq s \leq k\} \] is 
contained in the polyomino (with no restriction on the filling of the other cells). 
See Figure~\ref{fig:nechain} for  an illustration. We will call the ne-chains  of size $k$ 
shortly $k$-chains. Note that in a moon polyomino $\M$, $k$ 1-cells in a north-east direction satisfy 
the submatrix condition if and only if the corners $(i_1, j_k)$ and $(i_k, j_1)$ are contained in $\M$,
 which is equivalent to the whole rectangle determined by these corners being contained in $\M$. In an
 almost-moon polyomino, the submatrix condition is satisfied if and only if the vertices $(i_1, j_k)$ 
and $(i_k, j_1)$ either determine a rectangle which is completely contained in $\M$ or an almost-rectangle 
 with one exceptional row contained in $\M$. In the latter case, the exceptional row does not contain any
 elements from the ne-chain.

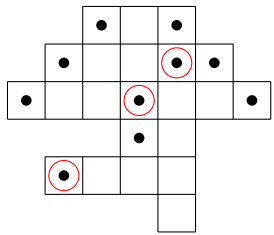
\begin{figure}[ht]
\centering
\begin{tikzpicture}
\draw (0, 1.5) -- (0,2);
\draw (0.5, 0.5) -- (0.5,1); \draw (0.5,1.5) -- (0.5,2.5);
\draw (1, 0.5) -- (1,1); \draw (1, 1.5) -- (1, 3);
\draw (1.5, 0.5) -- (1.5,3);
\draw (2, 0) -- (2,3);
\draw (2.5, 0) -- (2.5,3);
\draw (3, 1.5) -- (3,2.5);
\draw (3.5, 1.5) -- (3.5,2);

\draw (2,0) -- (2.5,0);
\draw (0.5,0.5) --(2.5,0.5);
\draw (0.5,1) --(2.5,1);
\draw (0,1.5) --(3.5,1.5);
\draw (0,2) --(3.5,2);
\draw (0.5,2.5) --(3,2.5);
\draw (1,3) --(2.5,3);
\fill  (0.25,1.75) circle (2pt);  
\fill  (0.75,0.75) circle (2pt);  \draw[color=red] (0.75, 0.75) circle (.2);
\fill  (0.75,2.25) circle (2pt);
\fill  (1.25,2.75) circle (2pt);
\fill  (1.75,1.25) circle (2pt);
\fill  (1.75,1.75) circle (2pt); \draw[color=red] (1.75, 1.75) circle (.2);
\fill  (2.25,2.25) circle (2pt); \draw[color=red] (2.25, 2.25) circle (.2);
\fill  (2.25,2.75) circle (2pt); 
\fill  (2.75,2.25) circle (2pt); 
\fill  (3.25,1.75) circle (2pt); 
\end{tikzpicture}
\caption{A 01-filling $M$ of an almost-moon polyomino with $\ne(M)=3$. The 1's are represented by dots and 
the 0-cells are drawn empty. The circled dots form the only 3-chain in $M$.}
\label{fig:nechain}
\end{figure}

For a 01-filling  $M$ of an almost-moon polyomino, we denote by $\ne(M)$ the size of its largest ne-chains. 
 Suppose that $\M$ has $k$ rows and $\ell$ columns and let 
$\textbf{r} \in \mathbb{N}^k$ and $\textbf{c} \in \mathbb{N}^\ell$. We will denote by
 $ \F(\M)$ the set of all 01-fillings of $\M$, by $\F(\M, n)$ those fillings with exactly $n$ 1's, 
 and by $\F(\M, \textbf{r}, \textbf{c})$  the set of  all 01-fillings with row sums given by  $\textbf{r}$ 
and column sums given by $\textbf{c}$. Our first main result states that if $\M$ and $\N$ are two almost-moon
 polyominoes related by an interchange of adjacent rows (e.g. Figure~\ref{fig:moonexample}), then the 
statistic $\ne$ is equidistributed over the sets $\F(\M, \ast, \textbf{c})$ and $ \F(\N, \ast, \textbf{c})$ 
of fillings of $\M$ and $\N$, respectively, with fixed column sums but arbitrary row sums:
\[ \sum_{M \in \F(\M, \ast, \textbf{c})} q^{\ne(M)} = \sum_{M \in \F(\N, \ast, \textbf{c})} q^{\ne(M)}. \] More precisely, we
 have the following theorem.

\begin{thm} \label{thm1} Let $\M$ and $\N$ be two almost-moon polyominoes such that $\N$ can be obtained 
from $\M$ by an interchange of  two adjacent rows. In addition assume that $\M$ and $\N$ have no exceptional rows other than 
the swapped ones. Then there is a bijection  
\[ \phi_{\M, \N}: \F(\M) \longrightarrow \F(\N) \] that preserves the column sums of the fillings and such 
that $\ne(\phi_{\M, \N} (M)) = \ne(M)$ for $M \in \F(\M)$. Moreover, $\phi_{\N, \M} \circ \phi_{\M, \N} = 1_{\F(\M)}$.
\end{thm}

Let $\mathcal{M}$ be an almost-moon polyomino with $k$ rows, $\sigma \in S_{k}$, and suppose the polyomino $\sigma\M$ 
obtained by permuting the rows of $\M$ according to $\sigma$ is also an almost-moon polyomino. Note that $\sigma\M$
 can also be obtained  by a sequence of steps in which only two adjacent rows are interchanged. Moreover, the
 order of steps can be chosen so that the intermediate polyominoes are all almost-moon polyominoes with no exceptional rows
 other than the swapped ones. In other
 words, the set of all $\sigma\M$ which are almost-moon polyominoes is connected by transposition of adjacent 
rows. One way to see this is to note that one can reach the polyomino in which the row lengths are descending 
from top to bottom by starting from $\M$ and first moving its exceptional row down until there is no longer rows below it,  
then moving the shortest row of $\M$  to the bottom, the second shortest row to the second position from below, etc. Consequently, by 
composing the maps from Theorem~\ref{thm1} we get the following corollary.

\begin{cor} \label{cor1} Let $\mathcal{M}$ be an almost-moon polyomino with $k$ rows and $\ell$ columns, $\sigma \in S_k$ be
 a permutation of the row indices such that $\sigma\M$  is also an almost-moon polyomino. Let $\textbf{c} \in \mathbb{N}^\ell$.
 Then there is a bijection $\phi: \F(\M, \ast, \textbf{c}) \longrightarrow \F(\sigma\M, \ast, \textbf{c})$ such 
that $\ne(\phi (M)) = \ne(M)$ for $M \in \F(\M, \ast, \textbf{c})$. Moreover, the size of
 $\{ M: M \in \F(\M,n), \ne(M) =k \}$ depends only on the set of lengths of the columns of $\M$. 
\end{cor}

\begin{proof} To prove the second part of the statement, suppose $\M_1$ is the moon polyomino with descending 
row lengths that can be obtained by reordering the rows of $\M$ as in the discussion above. Suppose that 
the same procedure applied to the transpose $\M_1^t$ of $\M_1$ yields the polyomino $\M_2$ with descending 
row lengths. Then $\M_2$ is a Ferrers shape and it depends only on the set of column lengths of $\M$.
 Since the transpose of  ne-chains are also  ne-chains, it follows from the first part that the size 
of $\{ M: M \in \F(\M,n), \ne(M) =k \}$ also depends only on the set of lengths of the columns of $\M$. \end{proof}

The fact that
 $\left|\{ M: M \in \F(\M, \ast, \textbf{c}), \ne(M) =k \} \right|= \left|\{ M: M \in \F(\sigma \M, \ast, 
\textbf{c}), \ne(M) =k \} \right|$ if $\M$ and $\sigma \M$ are moon polyominoes as well as the second 
part of Corollary~\ref{cor1} was proved algebraically by Rubey~\cite{Rubey}. Therefore, our results 
provide a bijective proof of these facts and extend them to a larger set of polyominoes in which 
these properties hold. In Section~\ref{S:remarks} we discuss why this extension is in a sense the best possible.

As discussed in~\cite{Rubey}, one cannot hope to simultaneously preserve both $\textbf{r}$ and $\textbf{c}$,
 i.e.,  the natural generalization 
$\left|\{ M: M \in \F(\M, \textbf{r}, \textbf{c}), \ne(M) =k \} \right| = \left|\{ M: M \in \F(\sigma \M, 
\sigma \textbf{r}, \textbf{c}), \ne(M) =k \} \right|$ does not hold. However, our  second main result implies 
that $\ne$ can be preserved together with both the row and column sums if the fillings are restricted to
 have at most one 1 in each row. 

\begin{thm} \label{thm2} Let $\M$ and $\N$ be two almost-moon polyominoes such that $\N$ can be obtained 
from $\M$ by an interchange of two adjacent rows. In addition assume that $\M$ and $\N$ have no exceptional rows other than 
the swapped ones. If $\textbf{r} \in \{0,1\}^*$ and $\mathbf{c} \in \mathbb{N}^*$, then there
 is a map \[ \psi_{\M, \N}: \F(\M, \textbf{r}, \textbf{c}) \longrightarrow \F(\N, \textbf{r}', \textbf{c})\] 
such that $\ne(\psi_{\M, \N} (M)) = \ne(M)$ for $M \in \F(\M, \textbf{r}, \textbf{c})$, where 
$\mathbf{r}'$ is obtained from $\mathbf{r}$ by exchanging the entries corresponding to the two swapped rows. 
\end{thm}

By the same discussion after Theorem \ref{thm1}, we get the following corollary.

\begin{cor} \label{cor2} 
 Let $\mathcal{M}$ be an almost-moon polyomino with $k$ rows and $\ell$ columns, $\sigma \in S_k$ be
 a permutation of the row indices such that $\sigma\M$  is also an almost-moon polyomino. 
 Let $\textbf{r} \in \{0,1\}^k$ and $\mathbf{c} \in \mathbb{N}^l$.  Then there
 is a bijection $ \psi: \F(\M, \textbf{r}, \textbf{c}) \longrightarrow \F(\sigma \M, \sigma\textbf{r}, \textbf{c})$ 
such that $\ne(\psi (M)) = \ne(M)$ for $M \in \F(\M, \textbf{r}, \textbf{c})$. 
Moreover, the size of $\{ M: M \in \F(\M,\textbf{r}, \textbf{c}), \ne(M) =k \}$ depends only on the 
sequence  of lengths of the columns of $\M$. 
\end{cor}
\begin{proof} 
 To see the second part of the statement, note that exchanging adjacent rows allows us to transform $\M$ to $\M_1$, the 
 polyomino with descending row lengths, whose shape is determined by the sequence  of lengths of the columns of $\M$. 
 \end{proof}



\section{Maximal increasing sequences in 01-fillings with fixed total sum } 

In this section we will describe the maps $\phi_{\M, \N}$ and prove Theorem~\ref{thm1}. To this end, let
 $\M$ and $\N$ be two almost-moon polyominoes related by an interchange of two adjacent rows (e.g.
 Figure~\ref{fig:moonexample}). Assume that $\M$ and $\N$ have no exceptional rows other than the swapped ones. 
 If the swapped rows are of equal length then  $\M = \N$ and we define $\phi_{\M, \N}$ 
to be the identity map on $\F(\M)$. 

Otherwise, suppose the lengths of the swapped rows are not equal and let $\R_s$ and $\R_l$, respectively, be the 
shorter and longer rows. Note that $\M \backslash \R_s $ and
  $\N \backslash \R_s $ have no exceptional rows.  Let $\alpha, \beta, \gamma, \delta$ be the fillings of the 
regions of $\R_s$ and $\R_l$ in $M$ as depicted on the left side of Figure~\ref{fig:f_example}. Precisely,
 $\alpha$ is the filling of the shorter row $\R_s$, $\beta$ is the filling of the part of the longer 
row $\R_l$ which has the same column support as row $\R_s$, and $\gamma$ and $\delta$ are the fillings of 
the two ends of row $\R_l$ so that the whole filling of row $\R_l$ viewed as a binary string is a
 concatenation of $\gamma$, $\beta$, and $\delta$. Note that one of $\gamma$ and $\delta$ may be the
 empty string. First we define a map $f_{\M, \N} : \F(\M) \longrightarrow \F(\N)$ as follows. 

 \begin{defi}\label{f_definition}  Using the notation above, $f_{\M,\N}(M)$ is the filling of $\N$ in which 
\begin{enumerate}[(1)]
\item the filling of the rows in $\N$ other than $\R_s$ and $\R_l$ is the same as in $M$,
\item the filling of the shorter row $\R_s$ of $\N$ is $\beta$,
\item the filling of the longer row $\R_l$ of $\N$ is the concatenation of $\gamma$, $\alpha$, and $\delta$.
\end{enumerate}
\end{defi} 
See Figure~\ref{fig:f_example} for an illustration.

\begin{figure}[ht]
\centering
  \begin{tikzpicture}
  \draw (0.5, 2) -- (3.5,2);
   \draw (0.5, 2.5) -- (3.5,2.5);
    \draw (1.5, 3) -- (3,3);
    \draw (0.5, 2) --(0.5, 2.5);
   \draw (1.5, 2) -- (1.5,3);
   \draw (3, 2) -- (3, 3);
   \draw (3.5, 2) -- (3.5, 2.5);
   \draw (2.25, 2.75) node {$\alpha$};
    \draw (2.25, 2.25) node {$\beta$};
     \draw (1, 2.25) node {$\gamma$};
     \draw (3.25, 2.25) node {$\delta$};
   \draw[densely dotted] (0.5, 2) .. controls(0,1) and (0.5,0.5) .. (1.5,0.5) ;
\draw[densely dotted] (1.5,0.5) .. controls (2,0.5) and (2.25, 1)  .. (3,0.5);
\draw[densely dotted] (3,0.5) .. controls (3.5, 0.5) and (4, 1) ..  (3.5, 2);
\draw[densely dotted] (1.5,3) -- (0.5,3);
\draw[densely dotted] (0.5,3) .. controls (0, 3.5) and (0, 4) .. (1.25, 4.5);
\draw[densely dotted] (1.25,4.5) .. controls (2, 4.75) and (2.5, 4.5) .. (3, 4.25);
\draw[densely dotted] (3,4.25) .. controls (3.5, 3.75) and (4, 3.3) .. (3.5,3);
\draw[densely dotted] (3,3) -- (3.5,3);
\draw[->] (-0.75, 2.75) node[left] {$\mathcal{R}_s$} -- (0.25, 2.75);
\draw[->] (-0.75, 2.25) node[left] {$\R_l$} -- (0.25, 2.25);

\draw[|->, xshift=5cm] (-0.25, 2.5) -- (1.75, 2.5) node[midway,above] {$f_{\M,\N}$};

  \draw[xshift = 7.5cm] (0.5, 3) -- (3.5,3);
   \draw[xshift = 7.5cm] (0.5, 2.5) -- (3.5,2.5);
    \draw[xshift = 7.5cm] (1.5, 2) -- (3,2);
    \draw[xshift = 7.5cm] (0.5, 2.5) --(0.5, 3);
   \draw[xshift = 7.5cm] (1.5, 2) -- (1.5,3);
   \draw[xshift = 7.5cm] (3, 2) -- (3, 3);
   \draw[xshift = 7.5cm] (3.5, 2.5) -- (3.5, 3);
   \draw[xshift = 7.5cm] (2.25, 2.75) node {$\alpha$};
    \draw[xshift = 7.5cm] (2.25, 2.25) node {$\beta$};
     \draw[xshift = 7.5cm] (1, 2.75) node {$\gamma$};
     \draw[xshift = 7.5cm] (3.25, 2.75) node {$\delta$};
\draw[densely dotted, xshift=7.5cm] (0.5, 2) .. controls(0,1) and (0.5,0.5) .. (1.5,0.5) ;
\draw[densely dotted, xshift=7.5cm] (1.5,0.5) .. controls (2,0.5) and (2.25, 1)  .. (3,0.5);
\draw[densely dotted, xshift=7.5cm] (3,0.5) .. controls (3.5, 0.5) and (4, 1) ..  (3.5, 2);
\draw[densely dotted, xshift=7.5cm] (1.5,2) -- (0.5,2);
\draw[densely dotted, xshift=7.5cm] (0.5,3) .. controls (0, 3.5) and (0, 4) .. (1.25, 4.5);
\draw[densely dotted, xshift=7.5cm] (1.25,4.5) .. controls (2, 4.75) and (2.5, 4.5) .. (3, 4.25);
\draw[densely dotted, xshift=7.5cm] (3,4.25) .. controls (3.5, 3.75) and (4, 3.3) .. (3.5,3);
\draw[densely dotted, xshift=7.5cm] (3,2) -- (3.5,2);
\draw[->, xshift=7.5cm] (4.75, 2.75) node[right] {$\R_l$} -- (3.75, 2.75);
\draw[->, xshift=7.5cm] (4.75, 2.25) node[right] {$\mathcal{R}_s$} -- (3.75, 2.25);
\end{tikzpicture}
\caption{The fillings $M$ and $f_{\M, \N}(M)$ differ only in the rows $\R_s$ and $\R_l$.}
\label{fig:f_example}
\end{figure}
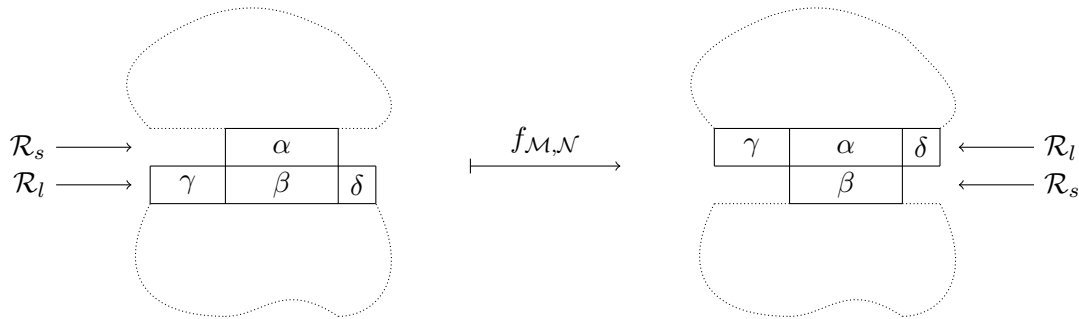

It follows directly from the definition.
\begin{lem}  \label{f_prop} $f_{\N,\M} \circ f_{\M,\N} =1_{\F(\M)}.$
\end{lem}


 If it is clear what the polyominoes $\M$ and $\N$ are, we will leave out the subscripts and write only $f(M)$. 

\begin{lem}\label{lem1} If the almost-moon polyominoes $\M$ and $\N$ are related by an interchange of the 
rows $\R_s$ ans $\R_l$ as above, then for every $M \in \F(\M)$, $$\left| \ne(M) - \ne(f_{\M, \N}(M)) \right| \leq 1.$$
\end{lem}

\begin{proof} Let $C$ be a $k$-chain in $M$.
\begin{itemize}
\item[(i)] If $C$ contains no 1-cells from $\R_s \cup \R_l$ then $C$ is a $k$-chain in $f(M)$.
\item[(ii)] If $C$ contains a 1-cell from $\alpha$ then $C$ is a $k$-chain in $f(M)$.
\item[(iii)] If $C$ contains a 1-cell $b_0$ from $\beta$ then $C-\{b_0\}$ is a $(k-1)$-chain in $f(M)$.
\item[(iv)] If $C$ contains a 1-cell from $\gamma$ or $\delta$ then $C$ is a $k$-chain in $f(M)$.
\end{itemize}
Therefore, $\ne(f(M)) \geq \ne(M)-1$. By switching the roles of $M$ and $f(M)$ with a similar analysis, we get $\ne(M) \geq \ne(f(M))-1$. 
\end{proof}
Therefore, every filling $M \in \F(\M)$ satisfies exactly one of the following 3 conditions:

\begin{center} $(\text{I}) \;\;\; \ne(f(M)) = \ne(M) \hspace{1cm} (\text{II}) \;\;\;\ne(f(M)) =
 \ne(M) +1 \hspace{1cm} (\text{III}) \;\;\; \ne(f(M)) = \ne(M) -1$.\end{center}
\noindent Let $\F^I(\M)$, $\F^{II}(M)$, and $\F^{III}(\M)$, be the fillings of $\M$ that satisfy the
 conditions~(I), (II), and (III), respectively.  Below we describe how $\phi_{\M,\N}(M)$ is defined on each of these three sets.

\fbox{\textbf{Case I.}} For $M \in \F^I(\M)$ we define $\phi_{\M, \N}(M) = f_{\M, \N}(M)$. It is clear
 that this defines a bijection from $\F^I(\M)$ onto $\F^I(\N)$.

\fbox{\textbf{Case II.}} For  $M \in \F^{II}(\M)$ with $\ne(M)=k$, we have  $\ne(N')=k+1$ where  $N'=f(M)$. 
 Reasoning as in the proof of Lemma~\ref{lem1}, one can see that all $(k+1)$-chains in $N'$ contain
 exactly one cell from $\R_s \cup \R_l$ and that that cell must be in the part $\alpha$ of the longer
 row $\R_l$ of $\N$ (see the right filling in Figure~\ref{fig:f_example}). 

\begin{defi}
 Suppose $M \in \F^{II}(\M)$ and $\ne(f(M)) = \ne(M)+1 = k+1$. The 1-cells in row $\R_l$ of $f(M)$
 which are  part of a $(k+1)$-chain are called \emph{problem cells}.
\end{defi} 

Let $\alpha_0$ be the set of problem cells in $N'=f(M)$ and let $\beta_0$ be the set of cells in row
 $\R_s$ of $N'$ that share a horizontal edge with the problem cells. 
 We define $\phi_{\M, \N}(M)$ to be the filling $N''$ of $\N$ obtained
 by replacing the problem cells in $\alpha_0$ by zeros and the cells in $\beta_0$ by ones. In other 
words, $\phi_{\M,\N}(M)=N''$ is obtained by a vertical shift of the problem cells in $N'$ from row $\R_l$ to row $\R_s$.

\fbox{\textbf{Case III.}} For  $M \in \F^{III}(\M)$, we have  $f_{\M,\N}(M) \in \F^{II}(\N)$. 
In this case we set $\phi_{\M, \N}(M) = f_{\M,\N}(\phi_{\N, \M}(f_{\M, \N}(M)))$. 

\vspace{.3cm}


In the following  we show that $\phi_{\M,\N}$ is a well-defined bijection from $\F(\M,*, \mathbf{c})$ to $\F(\N, *, \mathbf{c})$ 
preserving the statistic $\ne$. Lemmas \ref{beta0}--\ref{lem11}  show that $\phi_{\M,\N}$
 defined in Case II preserves column sum and $\ne(N'')=k$.  Lemma \ref{cyclelemma} shows that $\phi_{\M,\N}$,
 when restricted to $\F^{II}(\M)$, is a bijection onto $\F^{II}(\N)$. Lemma \ref{case3} shows $\phi_{\M,\N}$ 
 restricted to $\F^{III}(\M)$ is a bijection onto $\F^{III}(\N)$.

We first assume that $M$ is a filling in $\F^{II}(\M)$ with $\ne(M)=k$. 
\begin{lem} \label{beta0} 
If $a_0$ in row $\R_l$ of $f(M)$ is a problem cell then the cell $b_0$ in row $R_s$ of $f(M)$ in 
the same column as $a_0$ is a 0-cell.
\end{lem}

\begin{proof} Suppose $a_0$ is part of the $(k+1)$-chain $C$ in $f(M)$. If $b_0$ is a 1-cell then 
$(C-\{a_0\})\cup \{b_0\}$ is a $(k+1)$-chain in $M$. 
\end{proof}

Lemma \ref{beta0} shows that in Case II, the cells in $\beta_0$ are 0-cells for $N'$.  Hence  $N''$ is a 01-filling of $\N$ and 
the column sums of $M$ and $N''$ are the same. We will prove next that  $\ne(N'')=k$.

Since $\ne(N')=k+1$, from the way $N''$ is constructed it can readily be seen that $\ne(N'') \geq k$. 
Suppose $\ne(N'') > k$.  Since the ne-chains are preserved under rotation by $180^{\circ}$, in the discussion
 below we can assume, without loss of generality, that in $\M$ the shorter row $\R_s$ is above $\R_l$. This 
will allow us to avoid introducing too much notation, and instead use words such as ``below'', ``above'', etc.   Let $\S$ be the intersection of $\M$ (and therefore $\N$) and the vertical strip determined by the row $\R_s$. Precisely, let

$$\S =\{ (x, y): (x, y) \text{ is a cell in } \M \text{ such that there is a cell } (x, y') \in  \R_s \}.$$

\begin{lem} \label{chainN"} 
Every $(k+1)$-chain in $N''$ contains one 1-cell in each of the rows $\R_s$ and $\R_l$. Thus, 
every such chain is contained in the strip $\S$. 
\end{lem}

\begin{proof} Since $\ne(N'') > k$, $N''$ contains a $(k+1)$-chain $C$. If none of the elements of $C$ is 
in $\R_s \cup \R_l$, then $C$ is a chain in $M$, which yields a contradiction with $\ne(M)=k$. If $C$ contains
 a 1-cell $c$ from $\R_l$ but  not a cell from $\R_s$ then $C$ is a $(k+1)$-chain in $N'$ also, which means 
that $c$ is a problem cell, which contradicts the fact that during the construction of $N''$ all problem 
cells in $N'$ are replaced by 0's. If $C$ contains a 1-cell $c$ from $\R_s$ but  not a cell from $\R_l$ 
then there are 2 possible cases: (i) $c$ is a 1-cell in $N'$. In this case $C$ is a $(k+1)$-chain in $M$ 
as well. This contradicts $\ne(M)=k$. (ii) $c$ is a 0-cell in $N'$. In this case the cell $d$  above $c$ 
in row $\R_l$ is a problem 1-cell and $C$ is contained in the strip $\S$. Then $(C-\{c\}) \cup \{d\}$ is 
a $(k+1)$-chain in $M$. This again  contradicts   $\ne(M)=k$. \end{proof}

Let $y_1\text{---}  c \text{---} b \text{---} x_1$ be a $(k+1)$-chain in $N''$, where $c$ and $b$ are the 
1-cells in $\R_s$ and $\R_l$ respectively, $y_1$ is the part of the chain which is southwest of $c$ and $x_1$ 
is the part of the chain which is northeast of $b$. Note that, since $b$ is a 1-cell in $N''$, it is also
 a 1-cell in $N'$ but it is not a problem cell.

\begin{lem} \label{cainN'} 
The cell $c$ is a 0-cell in $N'$ and therefore the cell $a$ in row $\R_l$ directly above $c$ is  a problem cell. 
\end{lem}

\begin{proof} If $c$ is a $1$-cell in $N'$ then $y_1\text{---}  c \text{---} b \text{---} x_1$ is a $(k+1)$-chain 
in $N'$ as well, which implies that $b$ is a problem cell. But then $b$ would be a 0-cell in $N''$ and cannot be a part 
of a $(k+1)$-chain in $N''$.
\end{proof}

Therefore, the cell $a$ is a part of a $(k+1)$-chain $y_2 \text{---} a \text{---} x_2$ in $N'$. Moreover, this 
chain is not contained in the strip $\S$. Otherwise, $y_2 \text{---} a \text{---} x_2$ is a $(k+1)$-chain in $M$ 
as well. We will show that this implies a contradiction, as stated in Lemma \ref{lem11}.  
Figure~\ref{fig:twochains} illustrates the positions of the two chains.

\begin{figure}[ht]
\centering

  \begin{tikzpicture}

  \draw (0.5, 3) -- (3.5,3);
   \draw(0.5, 2.5) -- (3.5,2.5);
    \draw (1.5, 2) -- (3,2);
    \draw (0.5, 2.5) --(0.5, 3);
   \draw (1.5, 2) -- (1.5,3);
   \draw (3, 2) -- (3, 3);
   \draw (3.5, 2.5) -- (3.5, 3);
   \draw (2, 2.75) node[font=\small] {$a$};
   \draw (2.35, 2.78) node[font=\small] {$b$};
    \draw (2, 2.25) node[font=\small] {$c$};
    \draw[dashed] (1.5,0.5) -- (1.5,4.5);
    \draw[dashed] (3,0.5) -- (3,4.3);
     \draw (3, 1) node[left] {$\S$};
    \draw (1.6, 1)-- (2, 2) node[font=\small, midway, right] {$y_1$} ;
    \draw (0.6, 1.4) node[font=\small, below] {$y_2$} -- (2, 2) ;
    \draw (2.35, 3)-- (2.75, 4) node[font=\small, midway, right] {$x_1$} ;
    \draw (2, 3)-- (2.25, 4.4) node[font=\small, midway, left] {$x_2$} ;
\draw[densely dotted] (0.5, 2) .. controls(0,1) and (0.5,0.5) .. (1.5,0.5) ;
\draw[densely dotted] (1.5,0.5) .. controls (2,0.5) and (2.25, 1)  .. (3,0.5);
\draw[densely dotted] (3,0.5) .. controls (3.5, 0.5) and (4, 1) ..  (3.5, 2);
\draw[densely dotted] (1.5,2) -- (0.5,2);
\draw[densely dotted] (0.5,3) .. controls (0, 3.5) and (0, 4) .. (1.25, 4.5);
\draw[densely dotted] (1.25,4.5) .. controls (2, 4.75) and (2.5, 4.5) .. (3, 4.25);
\draw[densely dotted] (3,4.25) .. controls (3.5, 3.75) and (4, 3.3) .. (3.5,3);
\draw[densely dotted] (3,2) -- (3.5,2);
\draw[->] (4.75, 2.75) node[right] {$\R_l$} -- (3.75, 2.75);
\draw[->] (4.75, 2.25) node[right] {$\mathcal{R}_s$} -- (3.75, 2.25);

\draw (0,4) node[left]{$\N$};
\end{tikzpicture}
\caption{The $(k+1)$-chain $y_1\text{---}  c \text{---} b \text{---} x_1$   in $N''$ is contained in the 
strip $\S$. The $(k+1)$-chain $y_2 \text{---} a \text{---} x_2$ in $N'$ is not contained in $\S$. 
In $N': a=1, b=1, c=0$, in $N'': a= 0, b=1, c=1$.}
\label{fig:twochains}
\end{figure}
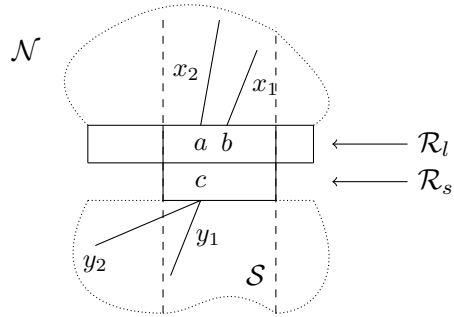

\begin{lem} \label{lem11} 
Suppose, using the notation above,
\begin{equation} y_1\text{---}  c \text{---} b \text{---} x_1 \text{ is a } (k+1)\text{-chain} \text{ in } N'' \text{ contained in } \S, \end{equation}
\begin{equation} \label{cond2} y_2 \text{---} a \text{---} x_2 \text{ is a } (k+1)\text{-chain} \text{ in } N' \text{ not contained in } \S. \end{equation}
Then $M$ has a $(k+1)$-chain or $N'$ has a $(k+1)$-chain that contains $b$.
\end{lem}

\begin{proof} We will discuss cases according to the relative positions of $x_1$ and $x_2$. For the
 discussion, it is helpful to visualize a chain as a piece of string directed northeast connecting
 the centers of its 1-cells, so that it's piecewise linearly increasing. Then we will say that chains
 cross if the corresponding strings cross. Let $r_{\M}(d)$ denote the row in $\M$ that contains the
 cell $d$ and define $r_{\M}(d_1) \leq r_{\M}(d_2)$ if the set of abscissas of the cells of row $r_{\M}(d_1)$
 is contained in the set of abscissas of the cells of $r_{\M}(d_2)$. For any almost-moon polyomino,
 this defines a total order on its rows.
Note that condition (2) implies that 
\begin{equation} \label{cond3} \R_s \leq r_{\N}(d) \text{ for every } d \in x_2 \cup y_2. \end{equation} 
Moreover,
\begin{equation} \label{cond4} \text{$C$ forms a $k$-chain if and only if  every two cells in $C$ form a 2-chain}, \end{equation} 
\begin{equation} \label{cond5} \text{$d_1 \in x_1$ and $d_2 \in x_2$ form a 2-chain if and only if they are in north-east direction}, \end{equation} 
\begin{equation} \label{cond6} \text{$d_1 \in y_1$ and $d_2 \in x_2$ such that $d_2$ is weakly to the left of a cell in $x_1$ form a 2-chain}, \end{equation} 
\begin{equation} \label{cond7} \text{$d_1 \in y_2$ and $d_2 \in x_1$ such that $d_2$ is not higher than the highest point of $x_2$ form a 2-chain}, \end{equation} 
Suppose first that both $x_1$ and $x_2$ are nonempty.
\begin{itemize} 
\item[(1)] If the chains $b \text{---} x_1$ and $a \text{---} x_2$ cross, let $P$ be the first intersection 
point. Let $x_3$ and $x_4$, respectively, be the parts of $x_1$ and $x_2$, respectively, that are strictly
 southwest of $P$. Then

- if $|x_4| > |x_3|$ then $y_1 \text{---} a \text{---}   x_4 \text{---} (x_1\backslash x_3)$ is at least a
 $(k+1)$-chain in $M$;\\
the submatrix condition is satisfied because of~\eqref{cond4}, \eqref{cond5}, and \eqref{cond6};

- if $|x_4| \leq |x_3|$ then $y_2 \text{---} b \text{---}  x_3 \text{---}  (x_2\backslash x_4)$ is at 
least a $(k+1)$-chain in $N'$;\\
the submatrix condition is satisfied because of~\eqref{cond4}, \eqref{cond5}, and \eqref{cond7}.

\item[(2)] If the chains $b \text{---} x_1$ and $a \text{---} x_2$ do not cross, there are 4 cases to be considered:
 \begin{itemize}
 \item[(a)] $x_2$ ends weakly below and strictly to the left of $x_1$. Let $x_3$ be the part of $x_1$ that
 is in the rows weakly below the highest cell of $x_2$.
 Then 
 
 - if $|x_2| > |x_3|$ then $y_1 \text{---} a \text{---}   x_2  \text{---} (x_1\backslash x_3)$ is at least
 a $(k+1)$-chain in $M$;\\
 the submatrix condition is satisfied because of~\eqref{cond4}, \eqref{cond5}, and \eqref{cond6};
 
 -  if $|x_2| \leq  |x_3|$  then $y_2  \text{---} b  \text{---} x_3$ is at least a $(k+1)$-chain in $N'$;\\
 the submatrix condition is satisfied because of~\eqref{cond4} and \eqref{cond7}.
 
 
  \item[(b)] $x_2$ ends strictly above and weakly to the left of $x_1$. Then 
  
  - if $|x_2| \leq |x_1|$ then $y_2  \text{---} b  \text{---} x_1$ is at least a $(k+1)$-chain in $N'$;\\
  the submatrix condition is satisfied because of~\eqref{cond4} and \eqref{cond7};
  
  - if $|x_2| >  |x_1|$ then $y_1  \text{---} a  \text{---} x_2$ is at least a $(k+1)$-chain in $M$;\\
  the submatrix condition is satisfied because of~\eqref{cond4} and \eqref{cond6}.
  
  \item[(c)] $x_2$ ends strictly above and  to the right of $x_1$. Let $x_4$ be the part of $x_2$ that is in
 the columns weakly to the left of the end of $x_2$. Then
  
  - if $|x_4| > |x_1|$ then $y_1  \text{---} a  \text{---} x_4$ is at least a $(k+1)$-chain in $M$;\\
  the submatrix condition is satisfied because of~\eqref{cond4} and \eqref{cond6};
  
  -  if $|x_4| \leq |x_1|$ then $y_2  \text{---} b  \text{---} x_1 \text{---} (x_2\backslash x_4)$ is at
 least a $(k+1)$-chain in $N'$;\\
  the submatrix condition is satisfied because of~\eqref{cond4}, \eqref{cond5}, and \eqref{cond7}.

   \item[(d)] $x_2$ ends weakly below and weakly to the right of $x_1$. This case is not possible because 
then $x_1$ and $x_2$ must cross.
 \end{itemize}
\end{itemize}
Finally, we consider the cases when one of $x_1$ and $x_2$ is empty.
\begin{itemize}
\item[(1)]  If  $|x_2|=0$  then $y_2  \text{---} b$ is a $(k+1)$-chain in $N'$; \\
 the submatrix condition is satisfied because of~\eqref{cond3}.

\item[(2)] If  $|x_2|\neq0$ but $|x_1|=0$,  let $x_5$ be the part of $x_2$ that is in the columns weakly to
 the left of $b$. Then

 - if $|x_5| > 0$ then $y_1  \text{---} a  \text{---} x_5$ is at least a $(k+1)$-chain in $M$;\\
  the submatrix condition is satisfied because of~\eqref{cond4} and \eqref{cond6};
  
 -  if $|x_5| =0 $ then $y_2  \text{---} b  \text{---} x_2 $ is a $(k+1)$-chain in $N'$;\\
  the submatrix condition is satisfied because of~\eqref{cond3}.
\end{itemize}
\end{proof}

To conclude, the construction of $N''$ and the preceding three lemmas imply the following property: 
 For $M \in \F^{II}(M)$, the above constructed filling $N'' = \phi_{\M, \N}(M) $ of $\N$ has the same
 column sums as $M$ and  $\ne(N'')=\ne(M)$.

We next prove some properties of  the transformation of the, so far, partially defined map $\phi_{\M,\N}$. 
These will be useful in extending the definition of  $\phi_{\M,\N}$ to  $\F^{III}(\M).$

\begin{lem}\label{cyclelemma} Let  $M \in \F^{II}(\M)$ with $\ne(M)=k$. 
 Let $N''=\phi_{\M,\N}(M)$ be defined as in Case II. Then 
\[ \ne(f_{\N,\M}(N'')) = \ne(N'') +1 = k+1.\] In other words, if $M \in \F^{II}(\M)$  then $\phi_{\M,\N}(M) \in \F^{II}(\N)$. 
Moreover, 
\[ \phi_{\N,\M}(\phi_{\M,\N}(M))=M. \] Consequently, $\phi_{\M,\N}$ restricted on $\F^{II}(\M)$ is a bijection onto $\F^{II}(\N)$.
\end{lem}
\begin{proof} Suppose $A$ is the set of 1-cells in row $\R_l$ of $N' = f(M)$ and  $A_0$ is the subset of $A$ 
containing the problem cells.  Let also $B$ be the set of 1-cells in $\R_s$ of  $N'$ and  $B_0$ the set of 
cells in $\R_s$ that share a horizontal edge with a cell in$A_0$. By construction, the set of 1-cells 
in row $\R_l$ of $N''=\phi_{\M,\N}(M)$ is $A_1= A \backslash A_0$ and the set of 1-cells in row $\R_s$ is $B_1 = B \cup B_0$. 

We now consider $M'=f_{\N, \M}(N'') \in \F(\M)$. Let $a \in A_0$ be a problem cell in $N'$. Then, by definition,
 $a$ is a part of a $(k+1)$-chain $y\text{---}a\text{---}x$ in $N'$ and the cells in the chain are all contained 
in the vertical strip determined by row $\R_l$. Let $b \in B_0$ be the neighbor cell of $a$. Then $y\text{---}b\text{---}x$ 
forms a $(k+1)$ chain in $M'$. Since $M'$ contains a $(k+1)$ chain,  Lemma~\ref{lem1} implies that $\ne(M') = k+1$. 
The preceding argument also shows that the 1-cells  in $B_0$ are problem cells in $M'$ and it is not difficult
 to see that these are the only problem cells. Therefore, to construct $\phi_{\N,\M}(N'') \in \F(\M)$, we 
replace the 1-cells in $B_0$ by 0's and the neighboring 0-cells in $A_0$ by 1's, which yields the initial filling $M$.
\end{proof}


\begin{lem}  \label{case3} 
Let $M \in \F^{III}(\M)$. Then $\phi_{\M, \N}(M)$ is well defined, its column sums 
are the same as $M$'s, and $\ne(\phi_{\M, \N}(M)) = \ne(M)$. Moreover, $\phi_{\M, \N}(M) \in \F^{III}(\N)$ and $\phi_{\N, \M} ( \phi_{\M, \N}(M)) =M$. Consequently, $\phi_{\M,\N}$ restricted on $\F^{III}(\M)$ is a bijection onto $\F^{III}(\N)$.
\end{lem}

\begin{proof} To see that $\phi_{\M, \N}(M)$ is well defined, let $N'= f_{\M, \N}(M)$. 
Then $f_{\N, \M}(N') = M$ and $N' \in \F^{II}(\N)$ because $\ne(f_{\N,\M}(N'))= \ne(N') +1$.
 Therefore, $\phi_{\N, \M}(N') $ is well defined and so is $\phi_{\M, \N}(M)$. 
 Additionally, Lemma~\ref{cyclelemma}  implies that $\phi_{\N, \M}(N') \in \F^{II}(\M)$. Therefore,
\[ \ne(\phi_{\M, \N}(M)) = \ne(\phi_{\N, \M}(f_{\M, \N}(M))) +1 = \ne(f_{\M, \N}(M)) +1 = \ne(M).\]
It is clear that the column sums of $M$ and $\phi_{\M, \N}(M)$ are equal. 
Finally, $\phi_{\N, \M} ( \phi_{\M, \N}(M)) =M$ follows from  Lemma~\ref{f_prop} and Lemma~\ref{cyclelemma}.
\end{proof}

\section{Maximal increasing sequences in fillings with restricted row sum} 


In this section we  restrict to fillings with at most one 1 in each row  and 
prove Theorem \ref{thm2}. Explicitly,  let $\M$ and $\N$ be  two almost-moon polyominoes that  can be  obtained 
from each  other by an interchange of two adjacent rows.  Assume that $\M$ and $\N$ have no exceptional rows other than the 
swapped ones. 
Let  $\textbf{r} \in \{0,1\}^*$ and $\mathbf{c} \in \mathbf{N}^*$, 
we shall construct  a bijection $\psi_{\M,\N}$  from 
 $\F(\M, \textbf{r}, \textbf{c}) $ to $\F(\M, \textbf{r}', \textbf{c})$ that preserves the
 size of the largest ne-chains, where $\textbf{r}'$ is obtained from $\mathbf{r}$ by exchanging the entries corresponding 
 to the two swapped rows. 

If the two swapped rows are of equal length, then $\M=\N$ and we can simply take $\psi_{\M,\N}$ to be the identity map. 
In the following, we assume that the two swapped rows are $\R_s$ and $\R_l$, where the length of $\R_s$ is smaller than that of $\R_l$. 
We keep the notations as in the previous section.  For any fillings $M $ 
let  $\alpha, \beta, \gamma, \delta$ be as defined in Figure  \ref{fig:f_example}.    Note that for 
 $M \in \F(\M, \textbf{r}, \textbf{c}) $, there is at most one 1 in $\alpha$, as well as in the 
union of $\beta$, $\gamma$ and $\delta$. 
 Define the \emph{coupling filling}  of $M$ to be the filling $M'$ of $\M$  which is obtained from $M$ by exchanging the fillings $\alpha$ and $\beta$. 
Let $N$ (resp. $N'$) be obtained from $M$ 
(resp. $M'$) by swapping the rows $\R_s$ and $\R_l$ together with their fillings. 
In other words, $N=f_{\M,\N}(M')$ and $N'=f_{\M,\N}(M)$. 
Clearly $N$ and $N'$ are coupling fillings of  each other.

We need the following lemma, which is the crucial observation for the construction of $\psi_{\M,\N}$. 
\begin{lem} \label{lemma3-1} 
 Let fillings $(M, M')$ be a pair of coupling fillings of in  $\F(\M, \textbf{r}, \textbf{c})$,  
 and $(N, N')=(f_{\M,\N}(M'), f_{\M,\N}(M))$ be  fillings of $\N$. Then  
$\ne(M) = \ne(N)$ or $\ne(M)= \ne(N')$. 
\end{lem} 

We postpone the proof of Lemma \ref{lemma3-1} and explain how it helps us  construct the bijection $\psi_{\M,\N}$.  Lemma \ref{lemma3-1} implies that 
\begin{eqnarray} \label{set-equal} 
\{ \ne(M), \ne(M')\} = \{ \ne(N), \ne(N')\}
\end{eqnarray} 
as multisets. To see this, note that if $\ne(N)=\ne(N')=k$, then applying Lemma \ref{lemma3-1} to both $M$ and $M'$ yields
 $\ne(M)=\ne(M')=k$. 
Otherwise, if $\ne(N)\neq \ne(N')$, then applying Lemma \ref{lemma3-1} to $N, N'$ yields  that one of $\ne(M), \ne(M')$ equals $\ne(N)$,
and the other equals $\ne(N')$. 
Equation \eqref{set-equal}  allows us to construct an ne-preserving  bijection between
 the coupling fillings $\{M, M'\}$ and  $\{N, N'\}$. 
Combining the bijections of all the coupling fillings we get a  desired bijection. 
Explicitly, we can describe the map  $\psi_{\M,\N}$ as follows:  

\noindent  Let $M$ be a filling in  $\F(\M, \textbf{r}, \textbf{c})$.   
\begin{enumerate} 
\item If either $\alpha$ or $\beta$ has no 1,  then let $\psi_{\M,\N}(M)= N$, 
the filling of $\N$ obtained by swapping the two rows together with their fillings.  
\item If both $\alpha$ and $\beta$ contain a 1 in the same column, then again let $\psi_{\M,\N}(M)= N$.
\item If each of $\alpha$ and $\beta$ contain a unique 1 in a distinct column, 
\[
\psi_{\M,\N}(M)=\begin{cases}
  N & \text{if } \ne(M)=\ne(N) \\
 N'  & \text{if } \ne(M) \neq \ne(N). 
 \end{cases}\]
\end{enumerate} 
It is clear that the map $\psi_{\M,\N}$ is well-defined and preserving the statistic $\ne$. In addition,  
$\psi_{\M,\N} $  and $\psi_{\N,\M}$ are inverse to each other.

\begin{proof}[Proof of  Lemma \ref{lemma3-1}] 
If either $\alpha$ or $\beta$ contains no 1, or they both contain a $1$ in the same column, 
then swapping the two rows with their fillings will not change any ne-chains, 
and hence $\ne(M)= \ne(N)$.  In the following we assume that each of $\alpha$ and $\beta$ contains a unique 1-cell in a 
distinct column, and  
there is no 1 in $\gamma$ or $\delta$. Furthermore we assume that in $\M$, the row $\R_s$ is above the row $\R_l$.  
The opposite case can be treated similarly.

We consider the relative positions of the two 1-cells in $R_s$ and $\R_l$.  See Figure \ref{1-cells} for an illustration. 

 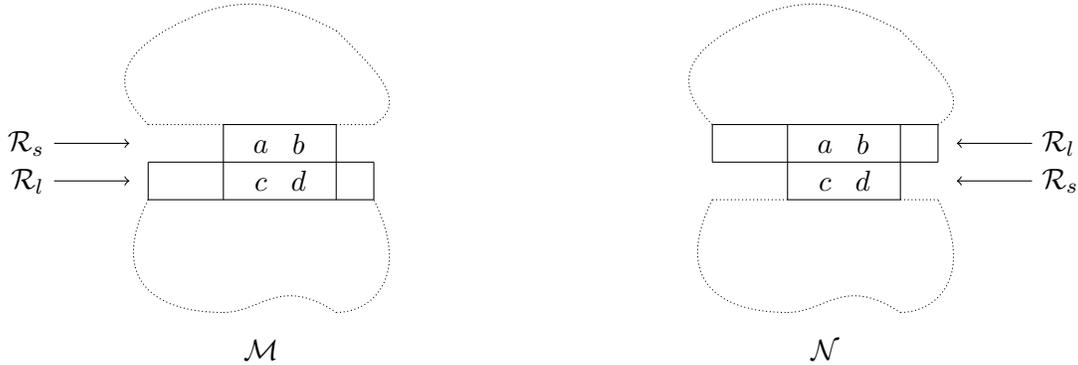
\begin{figure}[ht]
\centering
  \begin{tikzpicture}
  \draw (0.5, 2) -- (3.5,2);
   \draw (0.5, 2.5) -- (3.5,2.5);
    \draw (1.5, 3) -- (3,3);
    \draw (0.5, 2) --(0.5, 2.5);
   \draw (1.5, 2) -- (1.5,3);
   \draw (3, 2) -- (3, 3);
   \draw (3.5, 2) -- (3.5, 2.5);
   \draw (2.00, 2.70) node {$a$}; \draw (2.50,2.75) node {$b$};
    \draw (2.00, 2.20) node {$c$}; \draw (2.50,2.25) node {$d$};
   \draw[densely dotted] (0.5, 2) .. controls(0,1) and (0.5,0.5) .. (1.5,0.5) ;
\draw[densely dotted] (1.5,0.5) .. controls (2,0.5) and (2.25, 1)  .. (3,0.5);
\draw[densely dotted] (3,0.5) .. controls (3.5, 0.5) and (4, 1) ..  (3.5, 2);
\draw[densely dotted] (1.5,3) -- (0.5,3);
\draw[densely dotted] (0.5,3) .. controls (0, 3.5) and (0, 4) .. (1.25, 4.5);
\draw[densely dotted] (1.25,4.5) .. controls (2, 4.75) and (2.5, 4.5) .. (3, 4.25);
\draw[densely dotted] (3,4.25) .. controls (3.5, 3.75) and (4, 3.3) .. (3.5,3);
\draw[densely dotted] (3,3) -- (3.5,3);
\draw[->] (-0.75, 2.75) node[left] {$\mathcal{R}_s$} -- (0.25, 2.75);
\draw[->] (-0.75, 2.25) node[left] {$\R_l$} -- (0.25, 2.25);


  \draw[xshift = 7.5cm] (0.5, 3) -- (3.5,3);
   \draw[xshift = 7.5cm] (0.5, 2.5) -- (3.5,2.5);
    \draw[xshift = 7.5cm] (1.5, 2) -- (3,2);
    \draw[xshift = 7.5cm] (0.5, 2.5) --(0.5, 3);
   \draw[xshift = 7.5cm] (1.5, 2) -- (1.5,3);
   \draw[xshift = 7.5cm] (3, 2) -- (3, 3);
   \draw[xshift = 7.5cm] (3.5, 2.5) -- (3.5, 3);
   \draw[xshift = 7.5cm] (2.00, 2.70) node {$a$};  \draw[xshift=7.5cm]  (2.50,2.75) node {$b$};
    \draw[xshift = 7.5cm] (2.00, 2.20) node {$c$};\draw[xshift=7.5cm]  (2.50,2.25) node {$d$};
\draw[densely dotted, xshift=7.5cm] (0.5, 2) .. controls(0,1) and (0.5,0.5) .. (1.5,0.5) ;
\draw[densely dotted, xshift=7.5cm] (1.5,0.5) .. controls (2,0.5) and (2.25, 1)  .. (3,0.5);
\draw[densely dotted, xshift=7.5cm] (3,0.5) .. controls (3.5, 0.5) and (4, 1) ..  (3.5, 2);
\draw[densely dotted, xshift=7.5cm] (1.5,2) -- (0.5,2);
\draw[densely dotted, xshift=7.5cm] (0.5,3) .. controls (0, 3.5) and (0, 4) .. (1.25, 4.5);
\draw[densely dotted, xshift=7.5cm] (1.25,4.5) .. controls (2, 4.75) and (2.5, 4.5) .. (3, 4.25);
\draw[densely dotted, xshift=7.5cm] (3,4.25) .. controls (3.5, 3.75) and (4, 3.3) .. (3.5,3);
\draw[densely dotted, xshift=7.5cm] (3,2) -- (3.5,2);
\draw[->, xshift=7.5cm] (4.75, 2.75) node[right] {$\R_l$} -- (3.75, 2.75);
\draw[->, xshift=7.5cm] (4.75, 2.25) node[right] {$\mathcal{R}_s$} -- (3.75, 2.25);

\draw (2.00, 0) node {$\M$}; \draw[xshift=9.5cm] node {$\N$}; 
\end{tikzpicture}
\caption{Positions of the 1-cells in the coupling fillings. }
\label{1-cells}
\end{figure}

\fbox{\textbf{Case I.}} \ In the filling $M$ the two 1-cells in rows $\R_s$ and $\R_l$ form an ne-chain of size 2.  
That is, the 1-cells of  $\M$ in rows $\R_s$ and $\R_l$ are $b$ and $c$ as in the left figure of Figure \ref{1-cells} . 
Then $N$ is the filling of $\N$ in which the only 1-cells in rows $R_s$ and $R_l$ are $a $ and $d$, 
as in the right figure of Figure \ref{1-cells}. 
Clearly $\ne(N) \leq \ne(M)$ as every ne-chain of $N$ is still an ne-chain of $M$. 
 
Assume $\ne(M)=k$. If $\ne(N)\neq \ne(M)$ then $\ne(N)=k-1$, and every ne-chain of length $k$ in $M$ contains 
both $b$ and $c$.  In particular, $N'$, the coupling filling of $N$ which contains $b$ and $c$ in $\N$,
has ne-chains of length $k$. It follows that $\ne(N') \geq k$. 

Comparing fillings $N$ and $N'$. Let $C$ be an $l$-chain in $N'$. 
\begin{itemize}
\item[(i)] If $C$ contains no 1-cells from $\R_s \cup \R_l$ then $C$ is an $l$-chain in $N$.
\item[(ii)] If $C$ contains the 1-cell $c$,  then $C\cup \{ a\} -\{c\} $ is an $l$-chain in $N$.
\item[(iii)] If $C$ contains the 1-cell $b$, then $C-\{b\}$ is an $(l-1)$-chain in $N$.
\item[(iv)] If $C$ contains both $b$ and $c$,  then $C\cup \{a\}-\{b,c\}$ is an $(l-1)$-chain in $N$.
\end{itemize}
Therefore  $\ne(N) \geq \ne(N')-1$ and hence $\ne(N') \leq \ne(N) +1 \leq k$. Thus we must have $\ne(N')=k=\ne(M)$.

\fbox{\textbf{Case II.}}\  In the filling $M$ the two 1-cells in rows $\R_s$ and $\R_l$ do not form an ne-chain of size 2.  Then  $M$ has 1-cells 
$a, d$ in $\M$, $N$ has 1-cells $b,c$ and $N'$ has 1-cells $a,d$ in $\N$. 
Similarly as in Case 1 we have
\begin{eqnarray*}
\ne(M) \leq \ne(N) \leq \ne(M) +1 \\
\ne(N)-1 \leq  \ne(N') .
\end{eqnarray*} 
Using Lemma \ref{lem1} we also have $|\ne(M) - \ne(N')| \leq 1$. 
Assume that $\ne(M)=k$ is not equal to $\ne(N)$ or $\ne(N')$. 
Combining these three inequalities, we derive  that $\ne(N)=\ne(N')=k+1$. In addition, 
\begin{itemize} 
\item[(1)] Every $(k+1)$-chain  in $N$ contains both $b$ and $c$, and hence lies in the strip $\S$, where $\S$ is 
the union of columns that intersect the row $\R_s$, as defined in previous section.
\item[(2)] Every $(k+1)$-chain  in $N'$ contains $a$, and is not lying inside $\S$. 
\end{itemize} 
This is exactly the situation described in Lemma \ref{lem11}.  Comparing with Figure \ref{fig:twochains} and 
applying Lemma \ref{lem11}, also  noting that  the filling $N'$ does not have  $b$ as a 1-cell, we conclude  that $M$  has an ne-chain of length $k+1$. This is a contradiction. 
\end{proof}

Although Lemma \ref{lem11} plays an important role in the proofs of both Theorem \ref{thm1} and \ref{thm2},
we remark that the map $\psi_{\M,\N}$ is different than the map $\phi_{\M,\N}$ restricted to the set
$\F(\M,\mathbf{r}, \mathbf{c})$. For one thing, $ \phi_{\M,\N}$ does not always reserve the row sum,
hence not necessarily maps fillings of  $\F(\M,\mathbf{r}, \mathbf{c})$ to $\F(\N,\mathbf{r}', \mathbf{c})$ when 
$\mathbf{r} \in \{0,1\}^*$.

\section{Concluding remarks} \label{S:remarks}


We conclude this paper with some comments and counterexamples to a few seemingly natural generalizations of 
Theorem \ref{thm1} and \ref{thm2}. 

\noindent $\bullet$ \ 
\emph{Symmetry of $(\ne, \se)$ for fillings with $\mathbf{r} \in \{0,1\}^*$. }

 A \emph{southeast} chain, or shortly \emph{se-chain}, of size $k$ in a $01$-filling $M$  is a set of $k$ cells 
 \[\{ (i_1, j_k),  (i_2, j_{k-1}), \dots, (i_k, j_1)\}\] with 
 $i_1 < \cdots < i_k$, $ j_1 < \cdots < j_k$
 filled with 1's such that the $k \times k$ submatrix 
\[   \mathcal{G} = \{ (i_r, j_s) : 1 \leq r \leq k, 1 \leq s \leq k\} \] is 
contained in the polyomino.  We denote by $\se(M)$ the size of the largest se-chains of $M$.  
By the symmetry of $\ne$ and $\se$ we have that Lemma \ref{lemma3-1} also holds for $\se$. 

It is known that in 01-fillings of a Ferrers shape with fixed row sum and column sum in $\mathbf{N}$,
 the pair $(\ne, \se)$ may not distribute  symmetrically (e.g.  see \cite{Rubey,deMier06}). 
On the other hand,  when both $\mathbf{r}, \mathbf{c} \in \{0,1\}$, $(\ne, \se)$ does have a symmetric joint distribution. 
This was proved for Ferrers shapes by Krattenthaler \cite{Krattenthaler}, and implied for moon polyominoes by Rubey
\cite[Section 4.2]{Rubey}. 

The above results raised the question whether for almost-moon polyominoes the pair of  statistics $(\ne, \se)$ has 
a symmetric joint distribution when one or both of $\mathbf{r}, \mathbf{c}$ are in $\{0,1\}$, and whether 
the distribution of $(\ne, \se)$  is unchanged when one swaps two adjacent rows. 

The answers are all negative. In the following we give two set of counterexamples. 
The first one is for the case that $\mathbf{r} \in \{0,1\}$ but $\mathbf{c} \in \mathbf{N}$. 
The involved polyominoes are of small sizes, and we can list all the fillings explicitly.  
The second is for the case when both $\mathbf{r}$ and $\mathbf{c}$ are in $\{0,1\}$. 
We found the counterexample by running a computer program, and we will just describe the results without listing all the details.

\begin{exa} \label{exa:symmetry} 
Figures \ref{fig:not-sym0} and \ref{fig:not-sym}  list all the 01-fillings of three polyominoes, where the polyominoes in Figure \ref{fig:not-sym}  are obtained from 
the moon polyomino in Figure \ref{fig:not-sym0}  by moving down the first row. 
In all the fillings we require that   $\mathbf{r}=(1,1,1,1)$ and $\mathbf{c}=(2,1,1)$.  
The data $(\ne, \se)$ is given under each filling. 
\end{exa}

Figure \ref{fig:not-sym0} shows that even for moon polyomino with $\mathbf{r} \in \{0,1\}$ but $\mathbf{c} \in \mathbb{N}$, 
the distribution of $(\ne, \se)$ is not necessarily symmetric. 
Figure \ref{fig:not-sym} gives an example that  the distribution of the pair $(\ne, \se)$ is not  preserved when two adjacent rows are swapped in an almost-moon polyomino. 
Note that the first two fillings  $(M, M')$ in Figure \ref{fig:not-sym} are coupling fillings, 
whose corresponding coupling fillings are the first two fillings $(N, N')$ in the second row of  Figure \ref{fig:not-sym}.  
For these two pairs 
  Lemma \ref{lemma3-1} does not hold for $(\ne, \se)$, i.e., 
\[
\{ (\ne(M), \se(M)), (\ne(M'), \se(M') \} \neq \{ (\ne(N),\se(N)), (\ne(N'), \se(N')) \}. 
\]
 

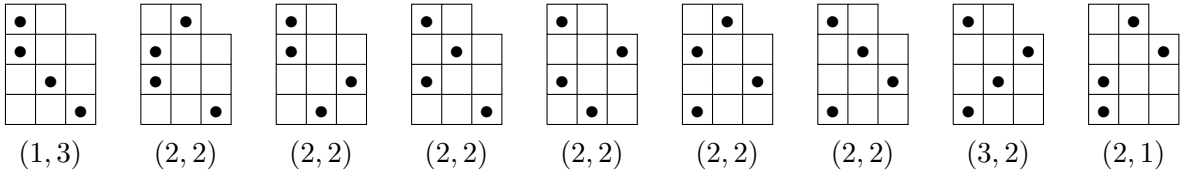
\begin{figure}[ht]
\begin{center}
\begin{tikzpicture}[scale=0.4]
\draw  (0,1)--(3,1)--(3,4)--(0,4)
                        (0,1)--(0,5)--(2,5)--(2,1)
                        (1,1)--(1,5) 
                         (0,2)--(3,2) (0,3)--(3,3) 
                        (1.5,0) node {$(1,3)$}
                        (0.5,4.4) node {$\bullet$} 
                         (0.5, 3.4) node {$\bullet$} 
                         (1.5,2.4) node {$\bullet$} 
                         (2.5,1.4) node {$\bullet$}; 

\draw[xshift=4.5cm]  (0,1)--(3,1)--(3,4)--(0,4)
                        (0,1)--(0,5)--(2,5)--(2,1)
                        (1,1)--(1,5) 
                         (0,2)--(3,2) (0,3)--(3,3) 
                        (1.5,0) node {$(2,2)$}
                        (0.5,2.4) node {$\bullet$} 
                         (0.5, 3.4) node {$\bullet$} 
                         (1.5,4.4) node {$\bullet$}                          (2.5,1.4) node {$\bullet$}; 

\draw[xshift=9cm]  (0,1)--(3,1)--(3,4)--(0,4)
                        (0,1)--(0,5)--(2,5)--(2,1)
                        (1,1)--(1,5) 
                         (0,2)--(3,2) (0,3)--(3,3) 
                        (1.5,0) node {$(2,2)$}
                        (0.5,4.4) node {$\bullet$} 
                         (0.5, 3.4) node {$\bullet$} 
                         (1.5,1.4) node {$\bullet$} 
                         (2.5,2.4) node {$\bullet$}; 

\draw[xshift=13.5cm]  (0,1)--(3,1)--(3,4)--(0,4)
                        (0,1)--(0,5)--(2,5)--(2,1)
                        (1,1)--(1,5) 
                         (0,2)--(3,2) (0,3)--(3,3) 
                        (1.5,0) node {$(2,2)$}
                        (0.5,2.4) node {$\bullet$} 
                         (0.5, 4.4) node {$\bullet$} 
                         (1.5,3.4) node {$\bullet$} 
                         (2.5,1.4) node {$\bullet$}; 

\draw[xshift=18cm]  (0,1)--(3,1)--(3,4)--(0,4)
                        (0,1)--(0,5)--(2,5)--(2,1)
                        (1,1)--(1,5) 
                         (0,2)--(3,2) (0,3)--(3,3) 
                        (1.5,0) node {$(2,2)$}
                        (0.5,2.4) node {$\bullet$} 
                         (0.5, 4.4) node {$\bullet$} 
                         (1.5,1.4) node {$\bullet$} 
                         (2.5,3.4) node {$\bullet$}; 

\draw[xshift=22.5cm]  (0,1)--(3,1)--(3,4)--(0,4)
                        (0,1)--(0,5)--(2,5)--(2,1)
                        (1,1)--(1,5) 
                         (0,2)--(3,2) (0,3)--(3,3) 
                        (1.5,0) node {$(2,2)$}
                        (0.5,1.4) node {$\bullet$} 
                         (0.5, 3.4) node {$\bullet$} 
                         (1.5,4.4) node {$\bullet$} 
                         (2.5,2.4) node {$\bullet$}; 

\draw[xshift=27cm]  (0,1)--(3,1)--(3,4)--(0,4)
                        (0,1)--(0,5)--(2,5)--(2,1)
                        (1,1)--(1,5) 
                         (0,2)--(3,2) (0,3)--(3,3) 
                        (1.5,0) node {$(2,2)$}
                        (0.5,1.4) node {$\bullet$} 
                         (0.5, 4.4) node {$\bullet$} 
                         (1.5,3.4) node {$\bullet$} 
                         (2.5,2.4) node {$\bullet$}; 

\draw[xshift=31.5cm]  (0,1)--(3,1)--(3,4)--(0,4)
                        (0,1)--(0,5)--(2,5)--(2,1)
                        (1,1)--(1,5) 
                         (0,2)--(3,2) (0,3)--(3,3) 
                        (1.5,0) node {$(3,2)$}
                        (0.5,1.4) node {$\bullet$} 
                         (0.5, 4.4) node {$\bullet$} 
                         (1.5,2.4) node {$\bullet$} 
                         (2.5,3.4) node {$\bullet$}; 

\draw[xshift=36cm]  (0,1)--(3,1)--(3,4)--(0,4)
                        (0,1)--(0,5)--(2,5)--(2,1)
                        (1,1)--(1,5) 
                         (0,2)--(3,2) (0,3)--(3,3) 
                        (1.5,0) node {$(2,1)$}
                        (0.5,2.4) node {$\bullet$} 
                         (0.5, 1.4) node {$\bullet$} 
                         (1.5,4.4) node {$\bullet$} 
                         (2.5,3.4) node {$\bullet$};

\end{tikzpicture}
\caption{Fillings of a moon polyomino with $\mathbf{r}= (1,1,1,1)$ and 
$\mathbf{c}= (2,1,1)$.  } \label{fig:not-sym0}
\end{center} 
\end{figure}

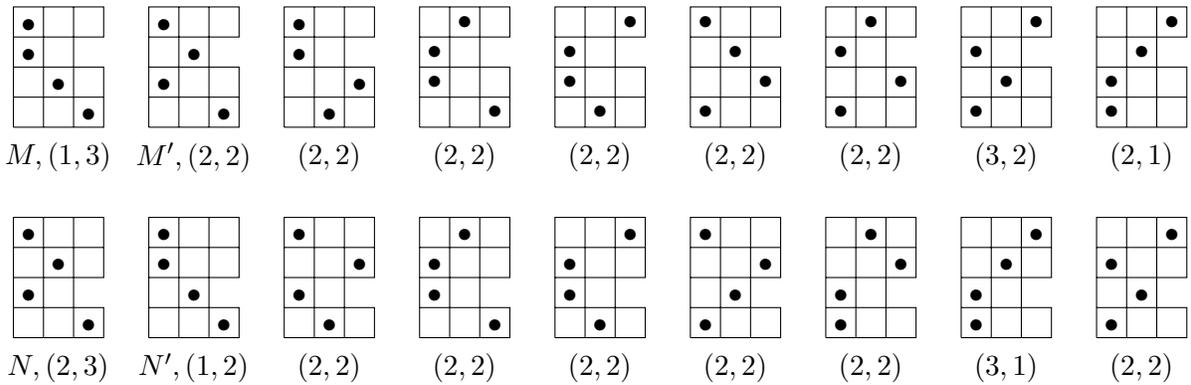
\begin{figure}[ht]
\begin{center}
\begin{tikzpicture}[scale=0.4]
\draw[yshift=7cm]  (0,1)--(3,1)--(3,3)--(0,3)
                        (0,1)--(0,5)--(3,5)--(3,4)--(0,4)
                        (1,1)--(1,5) (2,1)--(2,5)
                         (0,2)--(3,2) (1.5,0) node {$M,  (1,3)$}
                        (0.5,4.4) node {$\bullet$} 
                         (0.5, 3.4) node {$\bullet$} 
                         (1.5,2.4) node {$\bullet$} 
                         (2.5,1.4) node {$\bullet$}; 
                         
\draw [xshift=4.5cm,yshift=7cm] (0,1)--(3,1)--(3,3)--(0,3)
                        (0,1)--(0,5)--(3,5)--(3,4)--(0,4)
                        (1,1)--(1,5) (2,1)--(2,5)
                         (0,2)--(3,2) (1.5,0) node {$M', (2,2)$}
                        (0.5,4.4) node {$\bullet$} 
                         (0.5, 2.4) node {$\bullet$} 
                         (1.5, 3.4) node {$\bullet$} 
                         (2.5,1.4) node {$\bullet$}; 

\draw [xshift=9cm,yshift=7cm] (0,1)--(3,1)--(3,3)--(0,3)
                        (0,1)--(0,5)--(3,5)--(3,4)--(0,4)
                        (1,1)--(1,5) (2,1)--(2,5)
                         (0,2)--(3,2) (1.5,0) node {$(2,2)$}
                        (0.5,4.4) node {$\bullet$} 
                         (0.5, 3.4) node {$\bullet$} 
                         (1.5, 1.4) node {$\bullet$} 
                         (2.5,2.4) node {$\bullet$}; 

\draw[xshift=13.5cm,yshift=7cm] (0,1)--(3,1)--(3,3)--(0,3)
                        (0,1)--(0,5)--(3,5)--(3,4)--(0,4)
                        (1,1)--(1,5) (2,1)--(2,5)
                         (0,2)--(3,2)  (1.5,0) node {$(2,2)$}
                      (0.5,2.5)  node {$\bullet$} 
                       (0.5, 3.5)  node {$\bullet$} 
                        (1.5,4.5)  node {$\bullet$} 
                         (2.5,1.5)  node {$\bullet$} ;

\draw[xshift=18cm,yshift=7cm] (0,1)--(3,1)--(3,3)--(0,3)
                        (0,1)--(0,5)--(3,5)--(3,4)--(0,4)
                        (1,1)--(1,5) (2,1)--(2,5)
                         (0,2)--(3,2)  (1.5,0) node {$(2,2)$}
                      (0.5,2.5)  node {$\bullet$} 
                       (0.5, 3.5)  node {$\bullet$} 
                        (1.5,1.5)  node {$\bullet$} 
                         (2.5,4.5)  node {$\bullet$} ;

\draw[xshift=22.5cm,yshift=7cm] (0,1)--(3,1)--(3,3)--(0,3)
                        (0,1)--(0,5)--(3,5)--(3,4)--(0,4)
                        (1,1)--(1,5) (2,1)--(2,5)
                         (0,2)--(3,2)  (1.5,0) node {$(2,2)$}
                      (0.5,1.5)  node {$\bullet$} 
                       (0.5, 4.5)  node {$\bullet$} 
                        (1.5,3.5)  node {$\bullet$} 
                         (2.5,2.5)  node {$\bullet$} ;

\draw[xshift=27cm,yshift=7cm] (0,1)--(3,1)--(3,3)--(0,3)
                        (0,1)--(0,5)--(3,5)--(3,4)--(0,4)
                        (1,1)--(1,5) (2,1)--(2,5)
                         (0,2)--(3,2)  (1.5,0) node {$(2,2)$}
                      (0.5,1.5)  node {$\bullet$} 
                       (0.5, 3.5)  node {$\bullet$} 
                        (1.5,4.5)  node {$\bullet$} 
                         (2.5,2.5)  node {$\bullet$} ;

\draw[xshift=31.5cm,yshift=7cm] (0,1)--(3,1)--(3,3)--(0,3)
                        (0,1)--(0,5)--(3,5)--(3,4)--(0,4)
                        (1,1)--(1,5) (2,1)--(2,5)
                         (0,2)--(3,2)  (1.5,0) node {$(3,2)$}
                      (0.5,1.5)  node {$\bullet$} 
                       (0.5, 3.5)  node {$\bullet$} 
                        (1.5,2.5)  node {$\bullet$} 
                         (2.5,4.5)  node {$\bullet$} ;

\draw[xshift=36cm,yshift=7cm] (0,1)--(3,1)--(3,3)--(0,3)
                        (0,1)--(0,5)--(3,5)--(3,4)--(0,4)
                        (1,1)--(1,5) (2,1)--(2,5)
                         (0,2)--(3,2)  (1.5,0) node {$(2,1)$}
                      (0.5,1.5)  node {$\bullet$} 
                       (0.5, 2.5)  node {$\bullet$} 
                        (1.5,3.5)  node {$\bullet$} 
                         (2.5,4.5)  node {$\bullet$} ;

\draw[xshift=0cm] (0,1)--(3,1)--(3,2)--(0,2)
                        (0,1)--(0,5)--(3,5)--(3,3)--(0,3)
                        (1,1)--(1,5) (2,1)--(2,5)
                         (0,4)--(3,4) 
                         (0.5,4.4) node {$\bullet$} 
                         (0.5, 2.4) node {$\bullet$} 
                         (1.5,3.4) node {$\bullet$} 
                         (2.5,1.4) node {$\bullet$} 
                         (1.5,0) node {$N, (2,3)$};

\draw[xshift=4.5cm] (0,1)--(3,1)--(3,2)--(0,2)
                        (0,1)--(0,5)--(3,5)--(3,3)--(0,3)
                        (1,1)--(1,5) (2,1)--(2,5)
                         (0,4)--(3,4) 
                         (0.5,4.4) node {$\bullet$} 
                         (0.5, 3.4) node {$\bullet$} 
                         (1.5,2.4) node {$\bullet$} 
                         (2.5,1.4) node {$\bullet$} 
                         (1.5,0) node {$N', (1,2)$};      

\draw[xshift=9cm] (0,1)--(3,1)--(3,2)--(0,2)
                        (0,1)--(0,5)--(3,5)--(3,3)--(0,3)
                        (1,1)--(1,5) (2,1)--(2,5)
                         (0,4)--(3,4) 
                         (0.5,4.4) node {$\bullet$} 
                         (0.5, 2.4) node {$\bullet$} 
                         (1.5,1.4) node {$\bullet$} 
                         (2.5,3.4) node {$\bullet$} 
                         (1.5,0) node {$ (2,2)$};

\draw[xshift=13.5cm] (0,1)--(3,1)--(3,2)--(0,2)
                        (0,1)--(0,5)--(3,5)--(3,3)--(0,3)
                        (1,1)--(1,5) (2,1)--(2,5)
                         (0,4)--(3,4) 
                         (0.5,2.4) node {$\bullet$} 
                         (0.5, 3.4) node {$\bullet$} 
                         (1.5,4.4) node {$\bullet$} 
                         (2.5,1.4) node {$\bullet$} 
                         (1.5,0) node {$(2,2)$};

\draw[xshift=18cm] (0,1)--(3,1)--(3,2)--(0,2)
                        (0,1)--(0,5)--(3,5)--(3,3)--(0,3)
                        (1,1)--(1,5) (2,1)--(2,5)
                         (0,4)--(3,4) 
                         (0.5,2.4) node {$\bullet$} 
                         (0.5, 3.4) node {$\bullet$} 
                         (1.5,1.4) node {$\bullet$} 
                         (2.5,4.4) node {$\bullet$} 
                         (1.5,0) node {$(2,2)$};

\draw[xshift=22.5cm] (0,1)--(3,1)--(3,2)--(0,2)
                        (0,1)--(0,5)--(3,5)--(3,3)--(0,3)
                        (1,1)--(1,5) (2,1)--(2,5)
                         (0,4)--(3,4) 
                         (0.5,1.4) node {$\bullet$} 
                         (0.5, 4.4) node {$\bullet$} 
                         (1.5,2.4) node {$\bullet$} 
                         (2.5,3.4) node {$\bullet$} 
                         (1.5,0) node {$(2,2)$};      

\draw[xshift=27cm] (0,1)--(3,1)--(3,2)--(0,2)
                        (0,1)--(0,5)--(3,5)--(3,3)--(0,3)
                        (1,1)--(1,5) (2,1)--(2,5)
                         (0,4)--(3,4) 
                         (0.5,2.4) node {$\bullet$} 
                         (0.5, 1.4) node {$\bullet$} 
                         (1.5,4.4) node {$\bullet$} 
                         (2.5,3.4) node {$\bullet$} 
                         (1.5,0) node {$(2,2)$};

\draw[xshift=31.5cm] (0,1)--(3,1)--(3,2)--(0,2)
                        (0,1)--(0,5)--(3,5)--(3,3)--(0,3)
                        (1,1)--(1,5) (2,1)--(2,5)
                         (0,4)--(3,4) 
                         (0.5,1.4) node {$\bullet$} 
                         (0.5, 2.4) node {$\bullet$} 
                         (1.5,3.4) node {$\bullet$} 
                         (2.5,4.4) node {$\bullet$} 
                         (1.5,0) node {$(3,1)$};

\draw[xshift=36cm] (0,1)--(3,1)--(3,2)--(0,2)
                        (0,1)--(0,5)--(3,5)--(3,3)--(0,3)
                        (1,1)--(1,5) (2,1)--(2,5)
                         (0,4)--(3,4) 
                         (0.5,3.4) node {$\bullet$} 
                         (0.5, 1.4) node {$\bullet$} 
                         (1.5,2.4) node {$\bullet$} 
                         (2.5,4.4) node {$\bullet$} 
                         (1.5,0) node {$(2,2)$};

\end{tikzpicture}
\caption{Fillings of almost-moon polyominoes. The first row is for an almost-moon polyomino $\M$, 
and the second row is for $\N$ which differ from $\M$ by swapping the second and the third rows. } \label{fig:not-sym}
\end{center} 
\end{figure} 

\begin{exa} \label{almost-permutation} 
Our second example is restricted to 01-fillings where each row as well as each column has exactly one 1. 
We say that  such fillings are \emph{restricted}. 
Let $\M_1$, $\M_2$ and $\M_3$ be the polyominoes whose cells are given by
\begin{eqnarray*} 
\M_1 = \{ (i,j): \ 1 \leq i \leq 6, \  1 \leq j \leq 6\} -\{ (1,6)\}.  \\
\M_2=\{ (i,j): \ 1 \leq i \leq 6,  \ 1 \leq j \leq 6\} -\{ (1,5)\}.  \\
\M_3= \{ (i,j): \ 1 \leq i \leq 6, \  1 \leq j \leq 6\} -\{ (1,4)\}.  
\end{eqnarray*} 
Each polyomino has 600 restricted fillings. 
Let $G_1(x,y) = \sum_M x^{\ne(M) } y^{\se(M)}$ be the joint distribution of $(\ne,\se)$ over restricted fillings of $\M_1$. 
Similarly define $G_2(x,y)$ and $G_3(x,y)$ for restricted fillings in $\M_2$ and $\M_3$. 
With the help of a computer program we obtained 
\begin{align} \label{symm}  
G_1(x,y) =G_2(x,y)=  xy^5+x^5y + 72(x^2y^4+x^4y^2) + 48 (x^3y^4+x^4y^3) \nonumber \\ +  50(x^2y^3+x^3y^2) + 8(x^2y^5+x^5y^2) + 242 x^3y^3 
\end{align}
and 
\begin{align}  \label{no-symm} 
G_3(x,y) = xy^5+x^5y + 72x^2y^4+73x^4y^2 + 48x^3y^4+47x^4y^3 \nonumber \\+ 50 x^2y^3+49x^3y^2+ 8x^2y^5+8x^5y^2+ 243 x^3y^3.
\end{align}  
\end{exa} 
Equation \eqref{symm} is symmetric with respect to $x,y$, which is expected for $\M_1$ since it is a moon polyomino. 
Equation \eqref{no-symm} shows that the joint distribution of $(\ne, \se)$ over almost-moon polyominoes is not
necessarily symmetric, even if we require that every row and every column has exactly one 1. 
The difference between the two equations implies that the distribution of $(\ne, \se)$ may not be preserved
when two adjacent rows are swapped.

\noindent $\bullet$ \ \emph{Coupling fillings  with $\mathbf{r} \in \mathbb{N}^*$.  }

Another question is whether we can extend the idea of \emph{coupling} in Theorem \ref{thm2} to construct 
a bijection for Theorem \ref{thm1}. The following example shows that the direct application does not work.

\begin{exa} 
Consider fillings in $\F(\M,*, \mathbf{c})$ where a row may have multiple $1$s. 
For a filling $M$ of $\M$, we couple with it the filling $M'$ obtained from $M$ by swapping the fillings
$\alpha $ and $\beta$, just as what we did in Section 4. Again let  $N=f_{\M,\N}(M')$ and $N'= f_{\M,\N}(M)$. 
In the fillings shown in Figure \ref{fig:counter2}, $\ne(M)=3$, $\ne(M')=4$ while $\ne(N)=\ne(N')=4$. Lemma 
\ref{lemma3-1} does not hold for this case. 
\end{exa} 

\begin{figure}[ht]
\begin{center}
\begin{tikzpicture}[scale=0.4]
\draw
    (0,2)--(8,2)--(8,5)--(0,5)--(0,2)
    (0,6)--(8,6)--(8,8)--(0,8)--(0,6)
    (1,8)--(1,1)--(7,1)--(7,8)  
    (0,3)--(8,3)
    (0,4)--(8,4)
    (0,7)--(8,7)
    (2,1)--(2,8) (3,1)--(3,8) (4,1)--(4,8) (5,1)--(5,8) (6,1)--(6,8) 
    (0.5,2.4) node {$\bullet$} 
    (1.5,1.4) node {$\bullet$} 
    (2.5,5.4) node {$\bullet$} 
    (3.5,4.4) node {$\bullet$} 
    (4.5, 7.4) node {$\bullet$} 
    (5.5, 3.4) node {$\bullet$} 
    (6.5,5.4) node {$\bullet$} 
    (7.5,6.4) node {$\bullet$}
    (4,0) node {$M$};
    
\draw[xshift=10cm]
    (0,2)--(8,2)--(8,5)--(0,5)--(0,2)
    (0,6)--(8,6)--(8,8)--(0,8)--(0,6)
    (1,8)--(1,1)--(7,1)--(7,8)  
    (0,3)--(8,3)
    (0,4)--(8,4)
    (0,7)--(8,7)
    (2,1)--(2,8) (3,1)--(3,8) (4,1)--(4,8) (5,1)--(5,8) (6,1)--(6,8) 
    (0.5,2.4) node {$\bullet$} 
    (1.5,1.4) node {$\bullet$} 
    (2.5,4.4) node {$\bullet$} 
    (3.5,5.4) node {$\bullet$} 
    (4.5, 7.4) node {$\bullet$} 
    (5.5, 3.4) node {$\bullet$} 
    (6.5,4.4) node {$\bullet$} 
    (7.5,6.4) node {$\bullet$}  
    (4,0) node {$M'$};
\draw[xshift=10cm,red] (4.5,7.5) circle (10pt) 
                       (3.5,5.5)  circle (10pt) 
                       (2.5,4.5) circle (10pt) 
                       (1.5,1.5)  circle (10pt) ; 
\draw[xshift=20cm]
     (0,2)--(8,2)--(8,4)--(0,4)--(0,2)
    (0,5)--(8,5)--(8,8)--(0,8)--(0,5)
    (1,8)--(1,1)--(7,1)--(7,8)  
    (0,3)--(8,3)
    (0,6)--(8,6)
    (0,7)--(8,7)
    (2,1)--(2,8) (3,1)--(3,8) (4,1)--(4,8) (5,1)--(5,8) (6,1)--(6,8) 
    (0.5,2.4) node {$\bullet$} 
    (1.5,1.4) node {$\bullet$} 
    (2.5,4.4) node {$\bullet$} 
    (3.5,5.4) node {$\bullet$} 
    (4.5, 7.4) node {$\bullet$} 
    (5.5, 3.4) node {$\bullet$} 
    (6.5,4.4) node {$\bullet$} 
    (7.5,6.4) node {$\bullet$}
    (4,0) node {$N$};
    \draw[xshift=20cm,red] (4.5,7.5) circle (10pt) 
                       (3.5,5.5)  circle (10pt) 
                       (2.5,4.5) circle (10pt) 
                       (1.5,1.5)  circle (10pt) ; 

\draw[xshift=30cm]
     (0,2)--(8,2)--(8,4)--(0,4)--(0,2)
    (0,5)--(8,5)--(8,8)--(0,8)--(0,5)
    (1,8)--(1,1)--(7,1)--(7,8)  
    (0,3)--(8,3)
    (0,6)--(8,6)
    (0,7)--(8,7)
    (2,1)--(2,8) (3,1)--(3,8) (4,1)--(4,8) (5,1)--(5,8) (6,1)--(6,8) 
   (0.5,2.4) node {$\bullet$} 
    (1.5,1.4) node {$\bullet$} 
    (2.5,5.4) node {$\bullet$} 
    (3.5,4.4) node {$\bullet$} 
    (4.5, 7.4) node {$\bullet$} 
    (5.5, 3.4) node {$\bullet$} 
    (6.5,5.4) node {$\bullet$} 
    (7.5,6.4) node {$\bullet$} 
    (4,0) node {$N'$};
 \draw[xshift=30cm,red] (7.5,6.5) circle (10pt) 
                       (6.5,5.5)  circle (10pt) 
                       (5.5,3.5) circle (10pt) 
                       (0.5,2.5)  circle (10pt) ; 

 \end{tikzpicture}
\caption{Coupling fillings in $\F(\M, \textbf{r}, \textbf{c})$ with $\mathbf{r} \in \mathbb{N}^*$ and $\mathbf{r} \in \{0,1\}^*$. 
The circled dots form  4-chains in the polyominoes.
} 
\label{fig:counter2} 
\end{center} 
\end{figure}
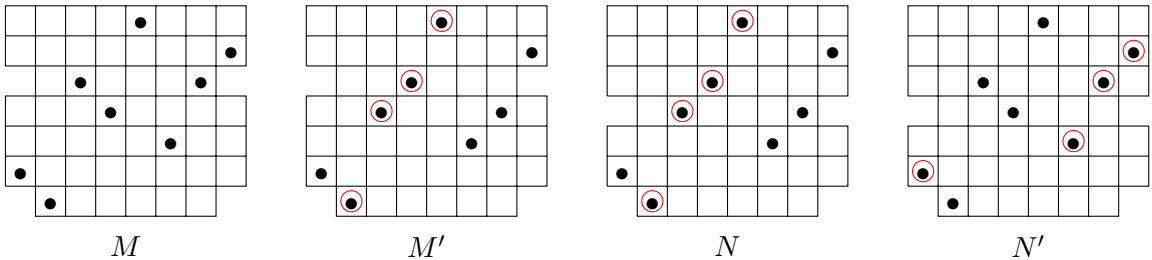

It is still open  whether Theorem 3 holds for the family of fillings that $\mathbf{r} \in \mathbb{N}^*$ but 
$\mathbf{c} \in \{0,1\}^*$  We point out that Lemma \ref{lemma3-1} does not hold in this case either.
Given a filling $M$ with multiple 1-cells in a row,  the natural way to define the coupling with the same row sum  
is to let  $M'$ be obtained from $M$ by keeping the empty columns of $\alpha \cup \beta$  and reversing 
the fillings in the remaining columns of  $\alpha \cup \beta$. 
Then $(N,N')$ are obtained from  $(M,M')$ by exchanging the two rows $\R_s$ and $R_l$  with their fillings. 
   
\begin{exa} 
The fillings in Figure \ref{fig:counter1}  gives an example where $\ne(M)=2, \ne(M')=3$ while $ \ne(N)= \ne(N')=3$.
\end{exa} 
\begin{figure}[t] 
\begin{center} 
\begin{tikzpicture}[scale=0.5]
\draw[xshift=0cm, yshift=7cm] (0,1)--(8,1)--(8,5)--(0,5)--(0,4)--(8,4)
(0,1)--(0,3)--(8,3) (0,2)--(8,2) (1,1)--(1,5) (2,1)--(2,5) (3,1)--(3,5) (4,1)--(4,5) (5,1)--(5,5) (6,1)--(6,5) (7,1)--(7,5) 
 (0.5,1.4) node {$\bullet$} 
 (4.5,1.4) node {$\bullet$} 
 (3.5,2.4) node {$\bullet$} 
 (6.5,2.4) node {$\bullet$} 
 (7.5, 2.4) node {$\bullet$} 
 (1.5,3.4) node {$\bullet$} 
 (5.5, 3.4) node {$\bullet$} 
  (2.5,4.4) node {$\bullet$} 
  (4,0) node  {$M$};

\draw[xshift=10cm, yshift=7cm] (0,1)--(8,1)--(8,5)--(0,5)--(0,4)--(8,4) (0,1)--(0,3)--(8,3) (0,2)--(8,2) (1,1)--(1,5) (2,1)--(2,5) (3,1)--(3,5) (4,1)--(4,5) (5,1)--(5,5) (6,1)--(6,5) (7,1)--(7,5) 
 (0.5,1.4) node {$\bullet$} 
 (4.5,1.4) node {$\bullet$} 
 (3.5,2.4) node {$\bullet$} 
 (6.5,2.4) node {$\bullet$} 
 (1.5, 2.4) node {$\bullet$} 
 (7.5,3.4) node {$\bullet$} 
 (5.5, 3.4) node {$\bullet$}  
  (2.5,4.4) node {$\bullet$} 
  (4,0) node  {$M'$}; 

\draw[xshift=0cm] (0,1)--(8,1)--(8,5)--(0,5)--(0,4)--(8,4) (0,1)--(0,2)--(8,2) (0,4)--(0,3)--(8,3) (1,1)--(1,5) (2,1)--(2,5) (3,1)--(3,5) (4,1)--(4,5) (5,1)--(5,5) (6,1)--(6,5) (7,1)--(7,5) 
 (0.5,1.4) node {$\bullet$} 
 (4.5,1.4) node {$\bullet$} 
 (1.5,2.4) node {$\bullet$} 
 (5.5,2.4) node {$\bullet$} 
 (6.5, 3.4) node {$\bullet$} 
 (3.5,3.4) node {$\bullet$} 
 (7.5, 3.4) node {$\bullet$} 
  (2.5,4.4) node {$\bullet$} 
  (4,0) node  {$N$}; 

\draw[xshift=10cm] (0,1)--(8,1)--(8,5)--(0,5)--(0,4)--(8,4) (0,1)--(0,2)--(8,2) (0,4)--(0,3)--(8,3) (1,1)--(1,5) (2,1)--(2,5) (3,1)--(3,5) (4,1)--(4,5) (5,1)--(5,5) (6,1)--(6,5) (7,1)--(7,5) 
 (0.5,1.4) node {$\bullet$} 
 (4.5,1.4) node {$\bullet$} 
 (5.5,2.4) node {$\bullet$} 
 (7.5,2.4) node {$\bullet$} 
 (6.5, 3.4) node {$\bullet$} 
 (1.5,3.4) node {$\bullet$} 
 (3.5, 3.4) node {$\bullet$} 
  (2.5,4.4) node {$\bullet$} 
  (4,0) node  {$N'$}; 

\end{tikzpicture}
\caption{Coupling fillings in $\F(\M, \textbf{r}, \textbf{c})$ with $\mathbf{r} \in \mathbb{N}^*$ and $\mathbf{c} \in \{0,1\}^*$.   } 
\label{fig:counter1} 
\end{center} 
\end{figure}
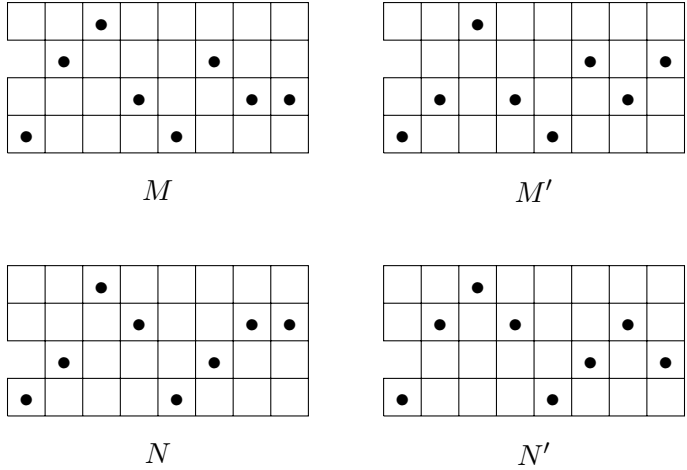

\noindent $\bullet$ \ 
 \emph{Statistic $\ne$ in fillings of general polyominoes} 
 
 It is natural to ask  whether Theorem \ref{thm1}  or Theorem \ref{thm2} can be extended to a more general family of polyominoes.  The following example shows that if there are more than one exceptional rows in the polyomino, then  the distribution of  $\ne(M)$ may not be the same after a swap of two adjacent rows.  

\begin{exa} 
This example was first given in \cite{PYY13} for \emph{layer polyominoes}, which are polyominoes that are row-convex and 
row-intersection-free.   The right one is almost-moon but the left one is not.  
Let $G(x, \F) = \sum_{M \in \F} x^{\ne(M)}$ be the generating function for the statistic $\ne$ over the fillings in a set $\F$. 
First consider  01-fillings
where every row and every column has
exactly one $1$.  We obtained the following generating functions. 
For the left polyomino $\M_1$,
\[
G(x, \F(\M_1, \mathbf{r}, \mathbf{c})) = x + 37x^2 + 31x^3 + 3x^4,
\]
 while for the right polyomino $\M_2$, 
\[
G(x, \F(\M_2, \mathbf{r}, \mathbf{c})) =  p + 36 x^2 +32 x^3+3 x^4,
\]
where $\mathbf{r}=\mathbf{c}=(1,1,1,1,1)$. 
This provides a counterexample of Theorem \ref{thm2} for general polyominoes. 
\end{exa} 

\begin{figure}[ht] 
\begin{center}
\begin{tikzpicture}[line width=0.7pt,scale=0.5]
\draw (5,2)--(0,2)--(0,0)--(5,0)--(5,5)--(0,5)--(0,4)--(5,4)
          (1,0)--(1,5) (2,0)--(2,5) (3,0)--(3,5) (4,0)--(4,5)
           (0,1)--(5,1) (1,3)--(5,3) ;
\draw (10,0)--(15,0)--(15,5)--(11,5)--(11,0) 
          (10,0)--(10,2)--(15,2) 
          (15,3)--(10,3)--(10,4)--(15,4)
          (10,1)--(15,1)
          (12,0)--(12,5) (13,0)--(13,5) (14,0)--(14,5)  
          ;
\node at (2.5,-1) {$\M_1$};
\node at (12.5, -1) {$\M_2$}; 
 \end{tikzpicture} 
 \end{center} 
\caption{Two general polyominoes related by an interchange of two adjacent rows.}  \label{fig:counter3} 
\end{figure}
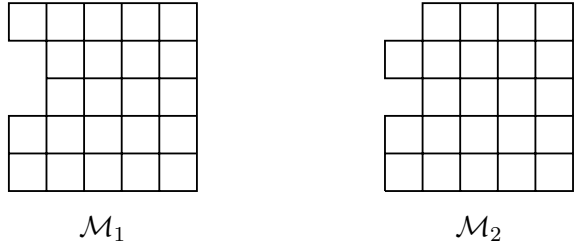 

The above example also implies that Theorem \ref{thm1} can not hold for general polyominoes. Namely, take the same two polyominoes and 
consider all the fillings in which each column has exactly one 1, but there is no constraint on the row.  
Note that 
\begin{itemize} 
 \item  Empty rows do not affect the statistic $\ne$ and can be ignored.
 \item  Any sub-polyomino of $\M_1$ or $\M_2$ containing three rows is an almost-moon polyomino. By Theorem \ref{thm1} rearranging rows 
for such three-row polyominoes does not change the distribution of $\ne(M)$ over $\F(\M, *, \mathbf{c})$. 
\end{itemize}
Now let  $\mathbf{c}=(1,1,1,1,1)$ 
we have 
\begin{enumerate} 
\item Either $G(\ne, \F(\M_1, *, \mathbf{c})) \neq G(\ne,  \F(\M_2, * , \mathbf{c}))$ and we have the desired counterexample, or 
\item $G(\ne, \F(\M_1, *, \mathbf{c})) = G(\ne, \F(\M_2, * , \mathbf{c}))$. But then the example from the previous paragraph 
implies that the distribution of $\ne$ over  01-fillings in $\F(\M_1, *, \mathbf{c})$ and $ \F(\M_2, * , \mathbf{c})$ with empty 
rows are different.
Note that 
the set of fillings of a polyomino $\M$  with empty rows can be obtained as the 
union of set $\F_i(\M, \mathbf{r}, \mathbf{c})$,  which consists of fillings 
in which the $i$-th row is empty. 
 An application of the inclusion-exclusion principle 
 implies that there is a sub-polyomino $\N_1$ of $\M_1$ consisting of 4 rows such that 
$G(\ne,  \F(\N_1, *, \mathbf{c})) \neq G(\ne,  \F(\N_2, * , \mathbf{c}))$ where $\N_2$ is the sub-polyomino of $\M_2$ 
consisting of the same four rows as in $\N_1$. 
\end{enumerate} 
  
\section*{Acknowledgment} 
We thank Hua Peng for writing a computer program to help us get the counterexamples in Example \ref{almost-permutation}.

\end{document}